\documentclass[noinfoline]{imsart}

\RequirePackage[OT1]{fontenc}
\RequirePackage{amsmath,amssymb,amsfonts,amsthm,bbm}
\RequirePackage{natbib}

\usepackage{enumerate}
\usepackage{amssymb}
\usepackage{mathrsfs}
\usepackage{amsthm}
\usepackage[utf8]{inputenc}
\RequirePackage[OT1]{fontenc}
\usepackage{enumerate}
\usepackage{amssymb}
\usepackage{amsfonts} 
\usepackage{amsmath}
\usepackage{accents}
\usepackage{color}
\usepackage{graphicx}
\usepackage{dsfont}
\usepackage{multirow}
\usepackage{mathtools}
\usepackage[usenames,dvipsnames,svgnames,table]{xcolor}
\usepackage{color,graphicx}
\definecolor{darkgreen}{rgb}{0,0.6,0}
\definecolor{red}{rgb}{0.7,0.15,0.15}
\definecolor{darkblue}{rgb}{0,0,0.6}
\usepackage{hyperref}
\usepackage{float}
\hypersetup{
	colorlinks=true,
	linkcolor=darkgreen,
	citecolor= darkblue,
	urlcolor = black}
\usepackage[capitalise]{cleveref}

\startlocaldefs

 \newtheorem{theorem}{Theorem}[section]
\newtheorem{lemma}[theorem]{Lemma}

\newtheorem{prop}[theorem]{Proposition}

\newtheorem{property}{Property}

\newtheorem{definition}[theorem]{Definition}
\theoremstyle{definition}
\newtheorem{ex}[theorem]{Example}

\newtheorem{defn}[theorem]{Definition}
\theoremstyle{remark}
\newtheorem{remark}[theorem]{Remark}
\newtheorem*{assumption*}{\assumptionnumber}
\providecommand{\assumptionnumber}{}
\makeatletter
\newenvironment{assumption}[2]
{%
	\renewcommand{\assumptionnumber}{Assumption (#1-#2)}%
	\begin{assumption*}%
		\protected@edef\@currentlabel{(#1-#2)}%
	}
	{%
	\end{assumption*}
}
\makeatother
\newcommand{\Pm}{\mathbb{P}}
\newcommand{\PS}{(\Omega, \mathcal{F},\left(\mathcal{F}_{t}\right)_{t\geq 0}, \mathbb{P})}
\newcommand{\ut}{\textbf{u}_{n}}
\newcommand{\uti}{\tilde{\textbf{u}}_{n}}
\newcommand{\f}{\widehat{\phi}_{n}(\uti)}
\newcommand{\fit}{\widehat{\phi}_{n}(\ut)}
\newcommand{\cl}{\phi_{n}(\ut)}
\newcommand{\clt}{\phi_{n}(\uti)}
\newcommand{\R}{\mathbb{R}}

\newcommand{\E}{\mathbb{E}}
\newcommand{\ch}{\mathds{1}}
\newcommand{\var}{\textsf{Var}}

\newcommand{\tv}{\textsf{TV}}

\newcommand{\class}{\mathcal{L}^{r}_{M}}
\newcommand{\estc}{\widehat{C}_{n}^{12}}
\newcommand{\cov}{C^{12}}

\newcommand{\cfx}{\psi_{n}(U)}
\newcommand{\cfxi}{\tilde{\psi}_{n}(U)}
\newcommand{\cfy}{\phi_{n}(U)}

\numberwithin{equation}{section}

\begin{document}
	
	\begin{frontmatter}
		\title{Minimax rates for the covariance estimation of multi-dimensional L\'evy processes with high-frequency data}
		\runtitle{Minimax rates for co-integrated volatility}
		
		\begin{aug}
			\author{\fnms{Katerina}  \snm{Papagiannouli}%\thanksref{t2}
				\ead[label=e1]{papagiai@hu-berlin.de}}
			
			%\thankstext{t2}{Supported by Deutsche Forschungsgemeinschaft via IRTG 1792 \textit{High Dimensional Nonstationary Time Series}.}
			
			\runauthor{K. Papagiannouli}
			
			\affiliation{Humboldt-Universit\"at zu Berlin}

		\address{ Institut f\"ur Mathematik\\Humboldt-Universit\"at zu Berlin\\ Unter den Linden 6\\10099 Berlin\\Germany\\ \printead{e1}}
			\end{aug}

		\begin{abstract}
			This article studies nonparametric methods to estimate the co-integrated volatility for multi-dimensional L\'evy processes with high frequency data. We construct a spectral estimator for the co-integrated volatility and prove minimax rates for an appropriate bounded nonparametric class of L\'evy processes. Given $ n $ observations of increments over intervals of  length $1/n$, the rates of convergence are $1 / \sqrt{n} $ if $ r\leq 1 $ and $ (n\log n)^{(r-2)/2} $ if $ r>1 $, where $ r $ is the co-jump index activity and corresponds to the intensity of dependent jumps. These rates are optimal in a minimax sense. We bound the co-jump index activity from below with the harmonic mean. Finally, we assess the efficiency of our estimator by comparing it with estimators in the existing literature. 
		\end{abstract}
		
		

	\end{frontmatter}

\section{Introduction}
L\'evy processes are the main building blocks for stochastic continuous-time jump models. Whenever the modeling of a stochastic process in finance requires the inclusion of jumps, L\'evy processes are those to be considered. They play an instrumental role, for example, in the modeling of financial data, see \cite{carr2002, barndorff2004, barndorff2006, wu2007, eberlein2005, geman2002}. \par

Consequently, the large amount of applications has given rise to a great demand for statistical methods in the study of L\'evy processes, especially nonparametric methods. Using nonparametric methods relaxes any dependency on the model. The problem of estimating the characteristics of a L\'evy process has received considerable attention over the past decade. Starting with the work by \cite{belomestny2006spectral}, a number of articles have considered nonparametric estimation methods for L\'evy processes. Therefore, one important task is to provide estimation methods for the characteristics of a L\'evy process. \par

Moreover, statistical methods require the nature of the observation schemes to be classified as high frequency or low frequency; here, we focus on a high frequency setting. If we can assume high-frequency observations for a L\'evy process, we can discretize a natural estimator based on continuous-time observations, where the jumps and the diffusion part are observed directly. In recent years, the literature on this subject has grown extensively, see \cite{figueroa2004nonparametric, todorov2011volatility, coca2018efficient, comte2009nonparametric, neumann2009nonparametric}. We now have vast amounts of data on the prices of various assets, exchange rates, and so on, typically \textit{tick data} which are recorded at every transaction time. \par

Much has been written on the estimation of L\'evy density using nonparametric techniques, for instance \cite{Nic}, \cite{duval2017compound}, \cite{comte2014nonparametric} and the references therein. However, we are interested in the estimation of the continuous part of a L\'evy  process, although jumps still play a central role in this estimation. In the univariate context, the seminal work of \cite{andersen1998answering} and \cite{barndorff2002}, proposed realized variance as an estimator for quadratic variation. In the presence of jumps, a well-known theoretical result proves that the realized variation converges in probability to the global quadratic variation as the time between two consecutive observations tends towards zero. This result motivated estimators that filter out the jumps, like Bipower Variation by \cite{barndorff2004} and Truncated Realized Variation by \cite{mancini2009}.

In the multivariate context, the recovery of co-integrated volatility (also known as covariance) becomes more complicated. Among various prominent works see \cite{pod1}, \cite{bibinger2015estimating}, \cite{bibinger2015econometrics}. For models incorporating jumps, the realized covariation converges in probability to the global quadratic variation containing the co-jumps. Co-jumps refer to the case when the underlying processes jump at the same time with the same direction. This raises the question how we can assess the dependent structure among the jump components. We find the answer in the L\'evy copula,  a subject studied by \cite{pt} in his PhD thesis. The interested reader should refer to \cite{tankov2003financial} and \cite{kallsen2006characterization}. The L\'evy Copula is the basic tool for the class of multidimensional L\'evy processes. To mention only the few approaches which are close to our focus on L\'evy processes we refer to \cite{mancini2017truncated}, \cite{pod1}, \cite{martin2017null}, \cite{bibinger2014estimating}, \cite{jacod2014remark}, \cite{bucher2013nonparametric} and \cite{MR3825892}.\par

Our aim in the present work is to provide minimax rates of convergence for the estimation of co-integrated volatility when the underlying process belongs to a certain class of multi-dimensional L\'evy processes. Many features of co-integrated volatility have already been studied, such as asynchronous observations, microstructure noise, and allowing for dependency among the jumps components. Whereas most of the aforementioned results prove central limit theorems for their estimators, at least to the best of our knowledge no work has dealt with optimal rates of convergence in the minimax sense. This work serves to fill this gap. \cite{jacod2014remark} proposed a spectral estimator for integrated volatility achieving minimax rates. In the present work, we generalize their work on finite dimensions. By virtue of simplicity, we will concentrate primarily on a two-dimensional regime, but extensions to the general multi-dimensional setting are straightforward to obtain as well. \par

For this purpose let us define, for a two-dimensional L\'evy process X, the Blumenthal- Getoor index $ r^{*}$:\\
\begin{equation}\label{bg}
B(r)= \int_{\R^{2}} \left(1\wedge \|\textbf{x}\|^{r}\right)F(d\textbf{x}), \quad I = \left\{0<r<2: B(r)< \infty \right\}, \quad r^{*}= \inf I
\end{equation} 
where $\textbf{x}= (x_{1}, x_{2})\in \R^{2} $ is the size of the jump components, $ \|\cdot\| $ is the Euclidean norm in $ \R^{2} $, $ F $ is the L\'evy measure. $ B(r) $ is not specifically interesting but the $ BG  $- index gives us the infimum number $ r$ for which $ B(r) $ is finite. This index is a very important number for the L\'evy processes, because using this index we can infer about the behavior of small jump components around $ 0 $. When we have a two-dimensional L\'evy process we have either independent jumps (i.e. disjoint or jumps in the axes) or dependent (i.e. co-jumps or joint jumps). In the present work we focus on the case of co-jumps, when the two marginals jump at the same time in the same direction. So the index $ r^{*}$ gives us information about the amount of disjoint and joint small jumps around $ 0 $. The behavior of co-jumps around $ 0 $, is described by 
\begin{equation}\label{cj}
\int_{\R^{2}}\left(1\wedge |x_{1}x_{2}|^{r/2}\right)F(dx_{1},dx_{2})< \infty.
\end{equation}
 \par
Here, we are interested in investigating the optimal rates for the estimation of co-integrated volatility when the model falls in a class of two-dimensional L\'evy processes, in case the jump components are either of finite or infinite variation and satisfy (\ref{cj}). Let $ X=(X^{(1)}, X^{(2)} ) $ and $ r_{1} $, $ r_{2} $ be the index of jump activity for the small jump components of each process $ X^{(1)}, X^{(2)} $ respectively. We find that $ r $, the index activity of co-jumps, is bounded from below by the harmonic mean of $ r_{1}, r_{2} $, even in the case of infinite variation jumps. This was not known up to now. Under this assumption for co-jumps we show that our spectral estimate for co-integrated volatility converges at a rate $(n\log n)^{\frac{(r-2)}{2}}$ if $r>1$ and $\frac{1}{\sqrt{n}}$ if $r \leq 1$. \par

Assuming a 2-dimensional It\^o semimartingale, \cite{mancini2017truncated} proposed a truncated covariance estimator to estimate co-integrated volatility at the rate $ \frac{1}{\sqrt{n}} $ when $ r_{1} $ is small and $ r_{2} $ is close to 1,
$n^{-\frac{1}{2}\big(1+\frac{r_{2}}{r_{1}}-r_{2}\big)}$ when $r_{1},r_{2}$ is much bigger than $ 1 $ and close to $ 2 $, $n^{\big(\frac{r_{2}}{2} - 1\big)}$ when $r_{1}$ is small and  $r_{2}$ is much bigger than $ r_{1} $ or in case of independent small jump components. However, these rates are sub-optimal for the class which we described in the last paragraph .\par

Let us describe the outline of this paper. In \Cref{sec:underlyingmodel} we state the underlying model. In \Cref{sec:assumptions} we give the assumptions to be satisfied in order to prove the minimax rates. In \Cref{theoretical_results} we construct our spectral estimator and state the results of this work. \Cref{sec:cojumpindex} gives the insight behind the co-jump index activity. In \Cref{sec:upperbound} we prove the upper bound for the family of our estimators.  In \Cref{sec:lowerbound} we present the proof of lower bound in a \textit{minimax} sense. We provide some comparison of our estimator with existent estimators in the literature in \Cref{sec:remarks}. In the last section we provide a simulation study. 

%%%%%%%%%%%%%%%%%%%%%%%%%%%%%%%%%%%%%%%%%%%%%%%%%%%%%%%%%%%%%%%%%%

\section{The underlying model} \label{sec:underlyingmodel}
We assume equidistant discrete observations with the consecutive time between two observations being $ i\Delta_{n}, i=0,\cdots,n $ for a mesh $ \Delta_{n}\to 0 $. Here, we use as a mesh $ \Delta_{n}=\frac{1}{n} $ and $ n\to \infty $. Regarding the time horizon of the process, it is observed on a finite time span $ [0,1] $. Let $ \textbf{X}=(X^{(1)}, X^{(2)})^\top $ be a two-dimensional L\'evy process with L\'evy-It\^o decomposition as 
\begin{equation}\label{2dimle}
\textbf{X}_{t}=\textbf{b}t+\textbf{W}_{t}+\int_{0}^{t}\int_{\|\textbf{x}\| \leq 1}\textbf{x}(\mu -\tilde{\mu})(ds,d\textbf{x})+ \int_{0}^{t}\int_{\|\textbf{x}\| > 1}\textbf{x}\mu(ds,d\textbf{x}).
\end{equation}
Unless stated otherwise, from now on $ \textbf{b} $ is a drift vector in $\R^{2} $, $ \textbf{W} =(W^{(1)},W^{(2)})$ denotes a bivariate Brownian motion with covariance matrix $ \Sigma\Sigma^{\top} $, and $ \mu, \tilde{\mu} $ are the jump measure and its compensator, respectively. The compensator takes the form $ \tilde{\mu}=dsF(d\textbf{x})$, where $F$ is the L\'evy  measure of $ \textbf{X} $.\par

Due to the independence of the continuous part and the discontinuous (jump part) of a L\'evy process, the analysis of $ \textbf{X} $ canonically splits into the inference on the covariance matrix and the inference on the jump measure $ F $. Our focus on this paper is to investigate an estimator for the co-integrated volatility of $ \textbf{X} $. \par 

We assume a filtered space $ \PS $ supporting two independent standard Brownian motions $ W^{(1)},W^{(3)} $ and two Poisson random measures $ \mu^{(j)} $ for $ j=1,2 $ on $ \R^{2}\times [0,1] $.  Recall that $ W^{(1)}, W^{(2)} $ are correlated with $ d\langle W^{(1)}, W^{(2)} \rangle_{t} \\ =\rho dt$, where $ \rho $ is a constant on $ [0,1] $. We construct $W^{(2)} $ as a linear combination of the two independent Brownian motions so $W^{(2)}_{t}=\rho W^{(1)}_{t}+\sqrt{1-\rho^{2}}W^{(3)}_{t}$. Next we calculate the variances and covariance of $W^{(1)}, W^{(2)} $, we see that the following holds 
\[
Var(W^{(1)}_{t})=\langle W^{(1)}_{t}, W^{(1)} _{t}\rangle =t
\]
\[
Var(W^{(2)}_{t})=\langle W^{(2)}_{t}, W^{(2)}_{t}\rangle =\rho^{2}t+(1-\rho^{2})t=t.
\]
For the covariance we obtain
\[
Cov(W^{(1)}_{t},W^{(2)}_{t})=\langle  W^{(1)}_{t}, W^{(2)}_{t}\rangle =\rho\langle W^{(1)}_{t}, W^{(1)}_{t}\rangle+\sqrt{1-\rho^{2}} \langle W^{(1)}_{t}, W_{t}^{(3)}\rangle=\rho t ;
\]
the last equality holds because of $ W^{(1)}, W^{(3)} $ being independent.
So, without loss of generality we assume that 
\[
\Sigma = \left(\begin{smallmatrix}\sigma^{(1)}& 0\\ \rho\sigma^{(2)} & \sqrt{1-\rho^{2}}\sigma^{(2)} \end{smallmatrix}\right) \quad \mbox{so that} \quad \Sigma\Sigma ^\top=\left(\begin{smallmatrix}(\sigma^{(1)})^{2} & \rho \sigma^{(1)}\sigma^{(2)}\\  \rho\sigma^{(1)}\sigma^{(2)}& (\sigma^{(2)})^{2} \end{smallmatrix}\right)
\] 
where $ \sigma^{(i)}, i=1,2 $ are deterministic. Therefore, the global quadratic variation of $ \textbf{X} $ is given by:
\begin{equation}
\langle X^{(1)}_{t}, X^{(2)}_{t}\rangle=\int_{0}^{t}\rho\sigma^{(1)}\sigma^{(2)} ds +\sum_{s\leq t}\Delta X^{(1)}_{s}\Delta X^{(2)}_{s}
\end{equation}
where the first term is the co-integrated volatility and the second term is the sum of products of simultaneous jumps (called co-jumps). Our target of inference, the co-integrated volatility at time $ 0\leq t \leq 1 $, is 
\begin{equation}
C^{12}_{t}=\int_{0}^{t}\rho \sigma^{(1)}\sigma^{(2)}ds.
\end{equation}

\section{Assumptions} \label{sec:assumptions}
To derive an estimator for co-integrated volatility and then prove minimax bound for this estimator, we need to establish some assumptions regarding the behavior of small jumps and the class of our estimator. In particular, our setup is intrinsically nonparametric and related to the properties of the observed path. We use the following notation for a matrix: $ \|\cdot\|_{\infty} $ is the maximum absolute row sum of the matrix, (i.e. the $ \infty $-norm).
\\

\begin{assumption}{1}{M}\label{as1}
The $ \infty $-norm of the covariance matrix is assumed to be bounded, i.e. $ \|\Sigma \Sigma ^{\top}\|_{\infty} \leq M $. 
\end{assumption}

\begin{assumption}{2}{M}\label{as2}
 $\int_{\R^{2}}\left(1\wedge |x_{1}x_{2}|^{r/2}\right)F(dx_{1},dx_{2})\leq M$ , where  $ r\in [0,2) $. 
\end{assumption}

Notice that Assumption \ref{as2} follows from the classical condition to control the activity of small jumps in two dimensions. Through a trivial calculation 

	\[
	\begin{aligned}
	\int_{\R^{2}}\left(1\wedge |x_{1}x_{2}|^{r/2}\right) F(dx_{1},dx_{2})
	& \leq \int_{\R^{2}}\left(1\wedge |x_{1}^{2}+x_{2}^{2}|^{r/2}\right)F(dx_{1},dx_{2})\\
	&=\int_{\R^{2}}\left(1\wedge ||\textbf{x}||^{r}\right) F(d\textbf{x}) .
	\end{aligned}
	\] 
By using this unconventional Assumption \ref{as2}, we relax the classical condition for small jumps in two dimensions and make our results stronger, since we consider the case of dependent jumps.\par
Assumption \ref{as2}  concerns the behavior of jump components with size smaller or equal to one. By this  assumption we consider the problem of controlling the activity of \textit{co-jumps}, i.e. joint jumps. Below, a co-jump, say at time $ t $, means that both components jump at this time but their jump sizes may \textit{not} be the same.  Ultimately, we are asking that is to say if the small jump components are of finite or infinite variation. This question concerns the behavior of the compensator $ F $, the L\'evy measure, near 0. The major difficulties here come form the possibly erratic behavior of $ F $ near $ 0 $ and the possible dependence between the jump components.  In \Cref{sec:cojumpindex} we describe more detailed the dependence structure of the jump components and the co-jump index activity $ r $.\par

The Blumenthal-Getoor (BG) index allows us to classify the processes from least active to most active, according to the above description. We denote by $ r^{*} $ the BG index for a two-dimensional L\'evy process which satisfies: 

\begin{equation}\label{bgtwo}
	r^{*}=\inf\left\{r \in[0,2):\int_{\R^{2}}\left(1\wedge ||\textbf{x}||^{r}\right) F(d\textbf{x}) <\infty\right\}.
\end{equation}
Note that a stable L\'evy process of index $ \beta \in(0,2)$ satisfies $\int_{\R^{2}}\left(1\wedge ||\textbf{x}||^{r}\right) F(d\textbf{x}) <\infty$ for all $r>\beta $, but not for $ r\leq \beta $. The BG index of a $ \beta $- stable is exactly $ \beta $.\par
The problem of BG index estimation from discrete observations of a L\'evy process has drawn much attention in the literature. In the case of high-frequency data, \cite{ait2011testing} studied the problem of estimating the jump activity index that is defined for any It\^o semimartingale. A consistent estimator for the BG index based on one-dimensional L\'evy processes with low-frequency data was obtained in \cite{belomestny2010spectral}. The interested reader should refer to \cite{belomestny2015estimation}, Section 7 for a detailed review of these results. An extension to time-changed L\'evy processes can be found in \cite{belomestny2012abelian, belomestny2013estimation}.

Now we will test the Assumption \ref{as2} about the boundedness of small co-jumps with some trivial examples. Despite its simple nature, the following example offers significant insight and intuitive understanding into co-jumps with infinite variation.
\begin{ex}\label{indjumps}
	Suppose we have independent jumps in the coordinates and $F\left(d\textbf{x}\right)$ is a L\'evy measure on $\R^{2}$ and $\textbf{x}$ is a vector in $\R^{2}$. Then, $supp (F)\subseteq \left\{\R\times \left\{0\right\}\cup \left\{0\right\}\times \R\right\}$ which means that
	\begin{equation*}
	\begin{aligned}
	\int_{\R^{2}}\left(1\wedge ||\textbf{x}||^{r}\right) F(d\textbf{x})&= \int_{\R^{2}}\left(1\wedge |x_{1}^{2}+x_{2}^{2}|^{r/2}\right)F(dx_{1},dx_{2})\\
	&= \int_{\R}\left(1\wedge|x_{1}|^{r}\right)  F_{1}(dx_{1})\\
	&+\int_{\R}\left(1\wedge |x_{2}|^{r}\right)  F_{2}(dx_{2})<\infty,
	\end{aligned}
	\end{equation*}
	if the marginals of a two-dimensional L\'evy process are finite in the one dimensional case, see the assumption section in \cite{jacod2014remark}. 
\end{ex}
In this example, we notice that 
\begin{equation}
\int_{\R^{2}}\left(1\wedge |x_{1}x_{2}|^{r/2}\right)(\ch_{\{x_{1}=0\}}+\ch_{\{x_{2}=0\}})F(dx_{1},dx_{2})=0,
\end{equation}
since the integrand is always equal to zero. This means that the deterministic error of our estimator is equal to zero, see \Cref{deterministicbou} for further details. This example shows us something more: Whenever we have independent jumps, no matter the choice of $F$, we can always find a control for the activity of small jumps. Even if we have jumps of infinite variation. 
\begin{ex}
	Suppose we have independent jump size distribution, which means that $F(dx)=F_{1}(dx_{1})F_{2}(dx_{2})$, so %Then, $\int_{\R^{2}}\left(1\wedge ||x||^{r}\right) F(dx)<\infty$.
	\begin{equation*}
	\begin{split}
	\int_{B_{1}(0)}|x_{1}x_{2}|^{r/2} F(dx_{1},dx_{2})&=\int_{B_{1}(0)}|x_{1}x_{2}|^{r/2} F_{1}(dx_{1})F_{2}(dx_{2})\\
	&\leq \int^{1}_{-1}\left(\int^{1}_{-1} |x_{1}|^{r/2} F_{1}(dx_{1})\right)|x_{2}|^{r/2} F_{2}(dx_{2}) <\infty
	\end{split}
	\end{equation*}
	if and only if $\int^{1}_{-1} |x_{i}|^{r/2} F_{i}(dx_{i})<\infty$ for $i=1,2$ and the Assumption \ref{as2} holds. 
\end{ex}
%%%%%%%%%%%%%%%%%%%%%%%%%%%%%%%%%%%%%%%%%%%%%%%%%%%%%%%%%%%%%%%%%%%%%%%%%%%%%%%%%%%%
\section{Theoretical results}\label{theoretical_results} 
We use standard notation for asymptotic quantities like $ X_{n} = O_{\Pm} (w_{n})$ if $ (X_{n}/w_{n})_{n\geq 1 } $ is stochastically bounded (i.e. bounded in probability or tight). 
We are in a nonparametric setting in which the process $ \textbf{X} $ belongs to the class $ \class $. Let us now define this class.
\begin{defn}\label{class}
For $M>0$ and $ r\in [0,2) $,  we define the class $\class$, the set of all L\'evy processes, satisfying
\begin{equation}\label{assu}
	\|C\|_{\infty}+\int_{\R^{2}}\left(1\wedge |x_{1}x_{2}|^{r/2}\right) F(dx_{1},dx_{2}) \leq M.
\end{equation}

\end{defn}
We adapt an estimator proposed by \cite{jacod2014remark}. Specifically, we let $\textbf{X}$ be a two-dimensional L\'evy process with characteristic triplet $ (\textbf{b}, C, F) $. Let us remember that we are in a high-frequency setting and the consecutive time between two observations is $\frac{1}{n}$. The characteristic function of $ \textbf{X}_{1/n} $ is given by:
\begin{equation}\label{char}
\begin{aligned}
\phi_{n}(\textbf{u}_{n})=\exp \bigg\{&\frac{1}{n}\bigg(i\left\langle \textbf{u}_{n},\textbf{b}\right\rangle-\frac{\left\langle C \textbf{u}_{n},\textbf{u}_{n}\right\rangle}{2}+\int_{\R^{2}}\big(\exp(i\left\langle \textbf{u}_{n}, \textbf{x}\right\rangle\big)\\
&-1-i\left\langle \textbf{u}_{n}, \textbf{x}\right\rangle \ch_{\left\{||\textbf{x}||_{\R^{2}}\leq 1\right\}}\big)F(d\textbf{x})\bigg)
\bigg\},
\end{aligned}
\end{equation}
where $C=\left(\begin{smallmatrix}C^{11} & C^{12}\\C^{21}&C^{22} \end{smallmatrix}\right)$ is the covariance matrix and $\textbf{u}_{n}=(U_{n}, U_{n})$. In the same vein, we define the characteristic function $ \phi_{n}(\tilde{\textbf{u}}_{n} )$ where $\tilde{\textbf{u}}_{n}=(U_{n}, -U_{n})$. Here, we focus on estimating the characteristic function on the diagonal of first and fourth quadrant for sake of simplicity to our calculations. The results still hold even when we move away from the diagonal. Following a trivial calculation we get that
\[
\left\langle C \ut, \ut \right\rangle= C^{11}U^{2}_{n}+C^{22}U^{2}_{n}+2C^{12}U^{2}_{n}
\]
\[
\left\langle C \uti, \uti \right\rangle= C^{11}U^{2}_{n}+C^{22}U^{2}_{n}-2C^{12}U^{2}_{n}.
\]
So, the covariance is given by
\begin{equation} \label{tr}
C^{12}=\frac{\left\langle C \ut , \ut \right\rangle-\left\langle C \uti, \uti \right\rangle}{4U^{2}_{n}}.
\end{equation}
We consider, based on the observations, the empirical characteristic function of the increments, at each stage n:
\begin{equation}\label{empi}
\fit=\frac{1}{n}\sum^{n}_{j=1}e^{i\left\langle \textbf{u}_n, \Delta^{n}_{j} \textbf{X}\right\rangle} \qquad \mbox{$ \textbf{u}_{n}\in\R^{2} $}. 
\end{equation}
Similarly, we consider the empirical characteristic function$ \f $.
Based on the trivial calculation (\ref{tr}), we now define the \textit{spectral estimator}
\begin{equation}\label{esti}
\widehat{C}_{n}^{12}(U_{n})=\frac{n}{2U^{2}_{n}}\left(\log \vert\hat {\phi}_{n}(\uti)\vert\mathds{1}_{\left\{\hat{\phi}_{n}(\uti)\neq 0\right\}}- \log \vert\hat {\phi}_{n}(\ut)\vert\mathds{1}_{\left\{\hat{\phi}_{n}(\ut)\neq 0\right\}}\right).
\end{equation}
The first result of this paper is the following theorem. 
\begin{theorem}\label{upbo}
	Let $ \textbf{X}  $ belong to the class $ \class $. Assume $M>0$ and $r\in[0,2)$, then as $ n\to \infty $ the family of estimators $\widehat{C}_{n}^{12}(U_{n})$ with 
	\[U_{n}=
	\begin{cases}
	\sqrt{n} \quad & \mbox{if $r\leq 1$}\label{vec}\\
	\sqrt{(r-1)n\log n}/\sqrt{M} \quad & \mbox{if $r>1$}
	\end{cases}
	\]
	satisfies $|\widehat{C}_{n}^{12}(U_{n})-C^{12}|= O_{\Pm}(w_{n})$ within the class $\class$ where
	\begin{equation}\label{rates}
		w_{n}=
		\begin{cases}
			1/\sqrt{n} \quad & \mbox{if $r\leq 1$}\\
			(n\log n)^{\frac{r-2}{2}}\quad & \mbox{if $r > 1$}.
		\end{cases}
	\end{equation}
	Particularly, we have that the family of estimators $\widehat{C}_{n}^{12}$ is consistent with the theoretical co-integrated volatility $C^{12}$ with the exact rates of convergence $w_{n}$.
\end{theorem}
$ \autoref{upbo} $ gives us an upper bound for the family of our estimators $\widehat{C}_{12}^{n} $. In \Cref{sec:upperbound}, we give a proof of the upper bound for the family of our estimators $ \widehat{C}_{12}^{n} $. Let us finally show that on the class $ \class $ the rate $ w_{n} $ (\ref{rates}) can be achieved exactly and thus constitutes the exact minimax optimal rate.
\begin{theorem}\label{lobo}
	Let $\textbf{X}$ belong to $\class$, $r\in[0,2)$ and $M>0$. Then there are constants $ A, B>0$ such that
	\[
	\liminf_{n\to \infty}\inf_{\widehat{C}_{n}^{12}}\sup_{C^{12}\in \class}\Pm[\text{d}(\widehat{C}_{n}^{12}, C^{12})> Aw_{n}]\geq B > 0,
	\]
	where $\widehat{C}_{n}^{12}$ is any estimator for the co-integrated volatility and $d$ is the euclidean distance on $\R^{2}$.
\end{theorem}
\autoref{lobo} gives us a lower bound for the family of our estimators $ \widehat{C}_{n}^{12} $ within the class $ \class $. The rates $w_{n}$ (\ref{rates}) for estimating $C^{12}$, namely the co-integrated volatility at time $ t=1 $ are optimal in a minimax sense. In \Cref{sec:lowerbound} we prove this result.

%%%%%%%%%%%%%%%%%%%%%%%%%%%%%%%%%%%%%%%%%%%%%%%%%%%%%%%%%%%%%%%%%%%%%%%%%%%%%%%%%%
\section{Co-jump index activity}\label{sec:cojumpindex}
We are interested in bounding from below the co-jump activity in the case that at least one of the jump components is of infinite variation. Each component $ X^{(i)} $ of a two-dimensional L\'evy process has its own index activity $ r_{i}$ for $ i=1,2 $. In the following we will describe the method for bounding from below the index activity of co-jumps. The BG index of a L\'evy process depends only on the L\'evy measure $ F$. $ r $ is an index taking care of positive and negative jumps, for simplicity's sake but without loss of generality we develop our method for the case in which the L\'evy measure is one-sided, i.e. $ X^{(i)} $ only makes positive jumps. Thus, $ r $  will be influenced by the dependent structure between the jump components.  \par

We will use a L\'evy copula to describe this dependency. The concept of L\'evy copula allows us to characterize in a time-independent scheme the dependence structure of the pure jump part of a L\'evy process. Here, we use the L\'evy copula, which permits a range from a dependent to a total independent framework. For the definition and concepts of independence and total positive dependence copula we refer to \cite{kallsen2006characterization}. The next definition is taken from \cite{mancini2017truncated}.
\begin{definition}\label{copula}
	The occurrence of joint jumps in $ (X^{(1)}, X^{(2)}) $ is described by the following tail integrals
	\[
	U(x_{1}, x_{2})=F_\gamma\big([x_{1},+\infty) \times [x_{2}, +\infty)\big)=C_{\gamma}\big(U_{1}(x_{1}), U_{2}(x_{2})\big), \quad \mbox{ $ x_{1}, x_{2}\in [0, \infty] $}
	\]
	where $ C_{\gamma}: [0, \infty]^{2}\to [0,\infty] $ is a L\'evy copula of the form
	\[
	C_{\gamma}(u_{1}, u_{2})=\gamma C_{\bot}(u_{1}, u_{2})+(1-\gamma)C_{\parallel}(u_{1}, u_{2}),
	\]
	where $ C_{\bot} = u_{2}\ch(u_{1}=\infty) +u_{1}\ch(u_{2}=\infty)$ is the independence copula, $ C_{\parallel}(u_{1}, u_{1})=u_{1}\wedge u_{2} $ is the total positive dependence copula and $ \gamma  $ varies in $ [0,1] $. 
\end{definition}
The following remark gives us some clarifications on the definition of the above L\'evy copula.
\begin{remark}
	The marginal tails $ U_{i} $ are defined on $ [0, \infty]$, the joint tail is defined on $ [0, \infty]^{2} $. $ u_{1}, u_{2} $ stands for $ U_{1}(x_{1}), U_{2}(x_{2}) $, and  $ (u_{1}, u_{2}) = (+\infty, +\infty) $ is allowed: both $ U_{i}(x_{i}) $ could be $ \infty $, namely when both $ x_{i}= 0 $. In that case $ U(x_{1}, x_{2}) = C_{\gamma}(U_{1}(x_{1}), U_{2}(x_{2}))$ is $ +\infty $, and $ C_{\gamma}(\infty, \infty) = 0$. $ C_{\gamma} $ is a L\'evy copula because it is a convex combination of two L\'evy copulas, i.e. $ C_{\gamma } $ is a $ 2- $increasing, grounded and with uniform margins, because $ C_{\bot} $ and $ C_{\parallel} $
	are such.  $ C_{\gamma} (u_{1}, u_{2}) $ is not a tail integral, it has different properties, for instance $ C_{\gamma}(u_{1}, +\infty) = u_{1}$, while for any tail $ U $ we have $ U(x_{1}, +\infty) = 0 $ Finally, when $ \gamma = 0  $ jump components are totally dependent while when $ \gamma =1 $ the opposite (ask Jacob).
	
\end{remark}
We observe that the index activity of co-jumps is bounded from below by the harmonic mean.

\begin{lemma}\label{coj}
	Suppose that Assumption \ref{as2} holds. Let $ X^{(i)} $ be an one-sided $r_{i}$-stable L\'evy process for $ i=1,2 $ with positive jumps. Given $ r_{1}, r_{2} \in [0,2)$ assume without loss of generality $ r_{1}\leq r_{2}$ and  $ r_{2}\geq 1 $. We assume either complete dependent or independent jumps. Then, we have that 
	\[
	r> \frac{2r_{1}r_{2}}{r_{1}+r_{2}}\geq r_{1}\wedge r_{2} 
	\]
	where $ r $ is the index activity of co-jumps.
\end{lemma}
\begin{proof}
	Each $ X^{(i)}$ is following a $r_{i}$-stable L\'evy process with L\'evy measure 
	\[
	F^{(i)}(dx_{i})=c_{i}x^{-1- r_{i}}_{i}\ch (x_{i}>0)dx_{i}
	\]
	for each $ i=1,2 $. We assume without loss of generality that $ c_{1}\leq c_{2} $. We denote by 
	\[
	U_{i}(x_{i}):=F^{(i)}\big([x_{i}, +\infty)\big)= c_{i}\frac{x_{i}^{-r_{i}}}{r_{i}} \quad \mbox{$ x_{i}\in [0, \infty] $}
	\]
	the tail integral of the marginal L\'evy measure $ F^{(i)} $. Note that $ r_{i} $ is the BG index of $ X^{(i)} $.  \par 
	The independent jumps have sizes of either $ (x_{1}, 0) $ or $ (0, x_{2}) $. This means that we have jumps only on the Cartesian axes. The independent copula regulates such jumps. On the other hand, the complete dependent jumps are regulated by the dependent copula; their size falls into the point $ (x_{1}, x_{2}) $. The complete dependent jumps are completely positively monotonic, i.e. there exists a strictly increasing and positive function $ f $ such that $ \forall t>0 $, $ \Delta X_{t}^{(2)}=f(\Delta X_{t}^{(1)}) $. This means that when $ x_{1} $ is a jump realisation so there is a realisation $ x_{2} $ such as $ x_{2}=f(x_{1}) $, then $ x_{1} $ is interpreted as the first component of the joint jump. In fact, the sizes $ (x_{1}, x_{2}) $ are supported by the graph $ x_{2}=f(x_{1}) $. For the dependent copula we need the minimum between $ U_{1}(x_{1}) $ and $U_{2}(x_{2}) $, which is attained when $ U_{1}(x_{1})=U_{2}(x_{2}) $. Hence, the graph $ x_{2}=U_{2}^{-1}(U_{1}(x_{1})) $ supports the joint jumps. \par
	In our case we assume one-sided $ r_{i} $-stable processes, which means that the union graph of the joint jumps is given by $ x_{2}=\left(\frac{c_{1}r_{2}}{r_{1}c_{2}} \right)^{-1/r_{2}}\cdot x_{1}^{r_{1}/r_{2}}$. We denote by $ F_{\gamma} $ the L\'evy measure in terms of the L\'evy copula, using the Definition \ref{copula}. Therefore,
	\begin{equation}\label{convexcop}
	\begin{aligned}
	F_{\gamma}(d\textbf{x})&=(1-\gamma)C_{\parallel}\big(U_{1}(x_{1}), U_{2}(x_{2})\big) +\gamma C_{\bot}\big(U_{1}(x_{1}), U_{2}(x_{2})\big).
	\end{aligned}
	\end{equation}
	Observe that $  \int 1\wedge (x_{1} x_{2})^{r/2}dC_{\bot}(U_{1}(x_{1}), U_2(x_{2})) = 0$, since the independent copula regulates the jumps on the axes. Inserting (\ref{convexcop}) into Assumption \ref{as2}, it turns out that for $ \epsilon  $ smaller than $ 1 $, we get
	\begin{equation}
	\begin{aligned}
	\int_{0\leq x_{1}, x_{2}\leq \epsilon}(x_{1} x_{2})^{r/2}F_{\gamma}(dx_{1}, dx_{2})&= \int _{0\leq x_{1}, x_{2}\leq \epsilon} \big(x_{1} x_{2}\big)^{r/2}dC_{\gamma}\big(U_{1}(x_{1}), U_{2}(x_{2})\big)\\
	&=(1-\gamma)\int _{0\leq x_{1}, x_{2}\leq \epsilon}\big(x_{1} x_{2}\big)^{r/2}dC_{\parallel}\big(U_{1} (x_{1}), U_{2}(x_{2})\big).
	\end{aligned}
	\end{equation}
	The first equality holds because of the fact that the integrand is always equally to zero in case of independent jumps. For sake of simplicity, we assume $ \gamma = 0 $, i.e. totally dependent jumps. We assume that the jump sizes $ (x_{1}, x_{2}) $ falls into the interval $ (0, \epsilon) $ for sufficiently small $ \epsilon>0 $.
	Remember $ r_{1}\leq r_{2} $ and $ c_{1}\leq c_{2} $, then we have $ U_{1}(\epsilon)\leq U_{2}(\epsilon) $, which implies $ \epsilon \geq U_{2}^{-1}\big( U_{1}(\epsilon)\big)=f(\epsilon)$. Since we want to bind $x_{1}\leq \epsilon $ and $ x_{2}=f(x_{1})\leq \epsilon $, this gives us $ x_{1}\leq f^{-1}(\epsilon)\wedge \epsilon=\epsilon $. Hence,
	\begin{equation}\label{indexco}
	\begin{aligned}
	\int_{0\leq x_{1}\leq \epsilon}\big(x_{1} \cdot f(x_{1})\big)^{r/2}dU_{1}(x_{1})&=\int_{0\leq x_{1}\leq \epsilon} \big(x_{1} \cdot f(x_{1})\big)^{r/2}c_{1}x_{1}^{-1-r_{1}}d x_{1}\\
	&=C^{r/2}\cdot c_{1}\int _{0\leq x_{1}\leq \epsilon} \Bigg( x_{1}^{\frac{r_{1}}{r_{2}}+1}\Bigg)^{r/2} x_{1}^{-1-r_{1}}dx_{1}
	\end{aligned}
	\end{equation}
	where $ C= \Big(\frac{c_{1}r_{2}}{r_{1}c_{2}}\Big)^{-\frac{1}{r_{2}}} $. \par
	In light of the above calculations, in order for the integral in (\ref{indexco}) not to be divergent we need $ \Big(\frac{r_{1}}{r_{2}}+1\Big)\frac{r}{2}-1-r_{1}>-1$, which means that $ r>\frac{2r_{1}r_{2}}{r_{1}+r_{2}} $. We observe that $ r $, the index activity of co-jumps, is at least the harmonic mean of the indices $ r_{1}, r_{2} $. In addition, $\frac{2r_{1}r_{2}}{r_{1}+r_{2}}\geq r_{1}$, since we assume $ r_{1}\leq r_{2} $. To conclude, the Blumenthal-Getoor (BG) index of the co-jump activity will be bounded from below by
	\begin{equation}\label{bgin}
	\quad r>\frac{2r_{1}r_{2}}{r_{1}+r_{2}}\geq r_{1} \wedge r_{2}.
	\end{equation}
	The proof now is complete.
\end{proof}
We see here that the higher the activity of one jump component, the higher the activity of co-jumps. \par
Next we proceed to the proof of the upper bound \autoref{upbo} using a spectral estimate for the co-integrated volatility. Given the fact that we know an estimate for the \textit{integrated volatility $ \widehat{IV} $}, we should consider a straightforward estimate for \textit{co-integrated volatility}. By polarization,
$
\widehat{IV}\left(X^{(1)}+X^{(2)}\right)/2-\widehat{IV}\left(X^{(1)}\right)/2-\widehat{IV}\left(X^{(2)}\right)/2, 
$
is a possible estimator for the co-integrated volatility. However, we refrain from using this estimate because the rates of convergence are slower than following the procedure as in \Cref{sec:upperbound}. Let us illustrate this argument with an example.
\begin{ex}\label{example}
	Let $ (X_{t}) \equiv (X_{t}^{(1)}, X_{t}^{(2)}) $ be a L\'evy process with characteristic triplet $ (0, 0, F(d\textbf{x})) $, i.e., without a Gaussian part. We assume its components are independent $ r_{i} -$ stable L\'evy processes for $ i=1,2 $ such that $ 0\leq r_{1}\leq r_{2} <2 $ and $ r_{2}\geq 1 $. Using Lemma 4.1 from \cite{kallsen2006characterization} F is supported by the coordinates axes and it can be written as $F(d\textbf{x})=F^{(1)}(dx_{1})+F^{(2)} (dx_{2})$. The L\'evy measures of the components are
	\[
	F^{(1)}(dx_{1})=\frac{1}{|x_{1}|^{1+r_{1}}}dx_{1}  \qquad \mbox{and} \quad F^{(2)}(dx_{2})=\frac{1}{|x_{2}|^{1+r_{2}}}dx_{2} .
	\]
	
	More precisely, 
	\begin{equation}
	\begin{aligned}
	\int 1 \wedge \|\textbf{x}\|^{r}F(d\textbf{x}) &= \int_{0<x_{1}, x_{2}<1}|x_{1}^{2}+x_{2}^{2}|^{r/2}F(dx_{1}, dx_{2})\\
	&= \int_{0<x_{1}<1}|x_{1}|^{r}F^{(1)}(dx_{1}) + \int_{0<x_{2}<1}|x_{2}|^{r}F^{(2)}(dx_{2})\\
	& = \int_{0<x_{1}<1}|x_{1}|^{r - 1 -r_{1}}dx_{1}+ \int_{0<x_{2}<1}|x_{2}|^{r-1 -r_2}dx_{2}\\
	\end{aligned}
	\end{equation}
	In order the integrals in the last equality not to be divergent we need $ r>r_{2} $ and $ r> r_{1} $. As a consequence, we find that $ r>\max(r_{2}, r_{1}) $. Using (\ref{bgtwo}) we find that the Blumenthal-Getoor index $r ^{*}= r_{2} $. 
	
\end{ex}

%%%%%%%%%%%%%%%%%%%%%%%%%%%%%%%%%%%%%%%%%%%%%%%%%%%%%%%%%%%%%%%%%%%%%%%%%%%%%%

\section{Upper Bound}\label{sec:upperbound}
 In this section we prove \autoref{upbo}. We say that a sequence of estimators $ \widehat{C}_{n}^{12} $ achieves the rate $w_{n} $ on $ \class  $, for estimating $ C_{12} $, if $|\widehat{C}_{n}^{12}-C_{12}| =O_{\Pm}(w_{n})$. This means that the family $\frac{1}{w_{n}}|\widehat{C}_{n}^{12}-C_{12}| $ is tight. Note that the argumentation in line is the bias-variance decomposition.
\subsection{The bias-variance decomposition}
We start with deriving a bias-variance-type decomposition of the estimation error of the estimator for cointegrated-volatility.
\begin{lemma}
	We have that 
	 \[
	 \widehat{C}_{n}^{12}(U_{n})-C_{12} = D_{n} + H_{n}.
	 \]
	 The \textit{deterministic error} given as
	 \begin{equation}\label{dete}
	 D_{n}=\frac{n}{2U^{2}_{n}}\bigg(\log \vert\phi_{n}(\uti)\vert-\log \vert\phi_{n}(\ut)\vert\bigg)-C_{12}
	 \end{equation}
	 and the \textit{stochastic error} as
	 \begin{equation}\label{st}
	 H_{n}=-\frac{n}{2U^{2}_{n}}\left(\log \Bigl|\frac{\phi_{n}(\uti)}{\phi_{n}(\ut)}\Bigr|-\log \Bigl|\frac{\f}{\fit}\Bigr|\bigg(\mathds{1}_{\left\{\f \neq 0 \quad \mbox{and}\quad \fit \neq 0\right\}}\bigg)\right).
	 \end{equation}
%	where $ \uti, \ut $ are characteristic functions of $ \textbf{X} $ and $\fit, \ft  $ are their empirical characteristic functions.  
	\end{lemma}

\begin{proof}
	We set $ C_{n}^{12} (U_{n})= \frac{n}{2U^{2}_{n}}\bigg(\log \vert\phi_{n}(\uti)\vert-\log \vert\phi_{n}(\ut)\vert\bigg) $, recalling the form of the estimator (\ref{esti}). We get
	\begin{equation}\label{eq1}
	\begin{aligned}
\widehat{C}_{n}^{12}(U_{n})-C_{12} & = \widehat{C}_{n}^{12}(U_{n}) +C_{n}^{12} (U_{n})-  C_{n}^{12}(U_{n}) -C_{12}\\
& =C_{n}^{12} (U_{n})-C_{12}+\widehat{C}_{n}^{12}(U_{n}) -  C_{n}^{12}(U_{n}) .
	\end{aligned}
	\end{equation}
Inserting (\ref{esti}) into (\ref{eq1}), we get that \textit{estimation error} is given by $\widehat{C}^{12}_{n}-C_{12}=D_{n}+H_{n}$. We need both quantities $ \fit, \f $ to be different than zero, otherwise the estimation error does not hold. 
\end{proof}

Our goal is to show that the estimation error is stochastically bounded, i.e. $ O_{\Pm}(w_{n}) $. Firstly, we bound the deterministic error.
\subsection{Bounding the deterministic error}\label{deterministicbou}
\begin{lemma}\label{deterr}
	Grant Assumption \ref{as2}. The deterministic error satisfies $|D_{n}|\leq \frac{M}{2}U^{r-2}_{n} + A U^{-2}_{n}$, where $ A$  is a positive constant.
\end{lemma}
\begin{proof}
	Recall the characteristic function of $ \textbf{X} _{1/n}$ in (\ref{char}). We define 
	\begin{align}\label{gamma}
	d_{n}&=2\int_{\R^{2}}\bigg(1-\cos\left(\big\langle \ut ,\textbf{x} \right\rangle\big)\bigg)F(d\textbf{x})\\
	\tilde{d}_{n}&=2\int_{\R^{2}}\bigg(1-\cos\big(\left\langle \uti,\textbf{x} \right\rangle\big)\bigg)F(d\textbf{x}). 
	\end{align}
	Therefore,
	\[
	\vert \phi_{n}(\ut)\vert=\exp\left(-\frac{1}{2n}\bigg(\left\langle C \ut, \ut \right\rangle+d_{n}\bigg)\right)
	\]
	and $\vert \phi_{n}(\uti)\vert=\exp\left(-\frac{1}{2n}\bigg(\left\langle C \uti,\uti \right\rangle+\tilde{d}_{n}\bigg)\right)$. Notice that here we use an argument of complex analysis. After taking the absolute value of the characteristic function, the imaginary part of the exponent is vanishing. Summing up,
	\begin{equation}\label{es}
	\frac{n}{2U^{2}_{n}}\big(\log \vert\phi_{n}(\uti)\vert-\log \vert\phi_{n}(\ut) \vert\big)- C_{12}=\frac{1}{4U^{2}_{n}}\left(d_{n}-\tilde{d}_{n}\right).
	\end{equation}
	By (\ref{es}), we have\\
	\begin{equation*}
	\begin{aligned}
	|D_{n}| &=\frac{1}{4U^{2}_{n}}\Bigg|\int\bigg(1-\cos\big(\left\langle \ut,\textbf{x}\right\rangle\big)\bigg)F(d\textbf{x})-\int\bigg(1-\cos\big(\left\langle\uti,\textbf{x}\right\rangle\big)\bigg)F(d\textbf{x})\Bigg|\\
	&=\frac{1}{4U^{2}_{n}}\Bigg|\int\bigg(\cos\big(\left\langle \uti ,\textbf{x}\right\rangle\big) -\cos\big(\left\langle \ut,\textbf{x}\right\rangle\big)\bigg)F(d\textbf{x})\Bigg|\\
	&\leq \frac{1}{4U^{2}_{n}}\int\left(2\wedge|\left\langle \ut,\textbf{x}\right\rangle^{2}-\left\langle\uti,\textbf{x}\right\rangle^{2}|\right)F(d\textbf{x})\\
	&=\frac{1}{4U^{2}_{n}}\int \big(2\wedge |4U^{2}_{n}x_{1}x_{2}|\big) F(dx_{1},dx_{2}),\\
	\end{aligned}
	\end{equation*}
	where we used the fact that $|\cos x - \cos y| \leq 2\wedge |x^{2}-y^{2}|$. Using the inequality $a\wedge b \leq a^{p}b^{1-p}$ for $p\in (0,1)$, the last term can be bounded as follows
	\begin{equation}
	\begin{split}
	|D_n| &\leq\frac{1}{2U^{2}_{n}}\int \bigg( 1\wedge 2U^{2}_{n}|x_{1}x_{2}|\bigg)F(dx_{1},dx_{2})\\
	&\leq \frac{2^{r/2}}{2U^{2}_{n}}\left(\int_{B_{1}(0)}\big(U^{2}_{n}|x_{1}x_{2}|\big)^{r/2}1^{1-r/2}F(dx_{1},dx_{2})+\int_{\R^{2}\setminus B_{1}(0)}1 F(dx_{1},dx_{2}) \right)\\
	&=\frac{2^{r/2}U^{r-2}_{n}}{2}\int_{B_{1}(0)}|x_{1}x_{2}|^{r/2}F(dx_{1},dx_{2})+
	\frac{U^{-2}_{n}}{2}F \big(\R^{2}\setminus B_{1}(0)\big),
	\end{split}
	\end{equation}
here $ \frac{r}{2}\in (0,1) $ because $ r\in(0,2) $ is the co-jump index activity. By Assumption \ref{as2} and for some constant $ A>0$,

	\begin{equation}\label{det}
	|D_{n}|\leq \frac{2^{r/2}M}{2}U^{r-2}_{n} + A U^{-2}_{n}
	\end{equation}
	as required.
\end{proof}

\subsection{Bounding the stochastic error}\label{stochasticbou}
We want to investigate how close the empirical characteristic function is to the characteristic function of a two-dimensional L\'evy process. The variables $e^{i\left\langle \ut, \Delta^{n}_{j}\textbf{X}\right\rangle}$ are i.i.d. as $j$ varies, with expectation $\phi_{n}(\ut)$. The same statement holds true for $e^{i\left\langle \uti,\Delta^{n}_{j}\textbf{X}\right\rangle}$ as well. So $\fit  $ is an unbiased estimator because $ \E[\fit] = \cl$. Also, the variance of the empirical characteristic function is given by $ \var(\fit) = \frac{1}{n}\left(1 - |\cl|^{2}\right) $.  
\begin{definition}\label{varco}
	For a $ \mathbb{C} $- valued random variable $ Z $ we define 
	\[
	\begin{aligned}
	\var(Z)&=\E[(Z-\E(Z))(\bar{Z}-\E(\bar{Z}))]\\
	&=\E[\bar{Z}Z -Z\E(\bar{Z})-\bar{Z}\E(Z)+\E(Z)\E(\bar{Z})]\\
	&=\E(|Z|^{2})-|\E(Z)|^{2}.
	\end{aligned}
	\]
\end{definition}
\begin{lemma}
	Let $V_{n}=\fit-\phi(\ut)$ and $\tilde{V}_{n}=\f-\phi_{n}(\uti)$, where $ V_{n}, \tilde{V}_{n} \in \mathbb{C}$. Then, $\E(|V_{n}|^{2})\leq \frac{1}{n}$ and $\E(|\tilde{V}_{n}|^{2})\leq \frac{1}{n}$.
\end{lemma}
\begin{proof}
	
	Set $V_{n}=Z\in \mathbb{C}$ such that $ V_n=\fit-\phi_{n}(\ut) $. Remember that $\widehat{\phi}_{n}$ is unbiased due to the fact that $\E[\fit]=\phi_{n}(\textbf{u}_{n})$, thus $|\E(Z)|^{2}=0$. Taking this into consideration with the previous definition (\ref{varco}), we obtain that $\E[|Z|^{2}]=\var[Z]$. Therefore,
	\begin{equation}\label{unvar}
	\E(|V_{n}|^{2})=\var(V_{n})= \var(\fit) = \frac{1}{n}\left(1 - |\cl|^{2}\right) \leq \frac{1}{n}.
	\end{equation}
	The same argument holds also for $ \E(|\tilde{V}_{n}|^{2}) $. This completes the proof.
	\end{proof}
We choose $ \ut=\big(U_{n}, U_{n}) $ and $ \uti=(U_{n}, -U_{n}) $. Recall that we estimate the characteristic function on the diagonal of first and fourth quadrant for calculation simplicity. Particularly, we choose for $ M>0 $, $ r\in[0,2) $ and $ n $ large enough
\begin{equation} \label{you}
U_{n}=
\begin{cases}
\sqrt{n} \quad & \mbox{if $r\leq 1$}\\
\sqrt{(r-1)n\log n}/\sqrt{M} \quad & \mbox{if $r>1$}.
\end{cases}
\end{equation}
\begin{lemma}\label{stoer}
	Grant Assumption \ref{as1}. For some positive constants $A,\Gamma, M $ and on the event $  \left\{|V_{n}|\leq \frac{1}{n^{r/4}}\right\}$ the stochastic error satisfies:
	\begin{equation}\label{stoch}
	\E\left[|H_{n}|\mathds{1}_{\left\{|V_{n}|\leq \frac{1}{n^{r/4}}\right\}}\right]\leq
	\begin{cases}
	\frac{A\Gamma}{\sqrt{n}} \quad & \mbox{if $r\leq 1$}\\
	\frac{4AM}{(r-1)n^{\frac{2-r}{2}}\log n }\quad & \mbox{if $r>1$}.
	\end{cases}
	\end{equation}
	
\end{lemma}
\begin{proof}
	Recalling the form of stochastic error (\ref{st}), the first quantity we need to bound is:
	\begin{equation}\label{boch}
	\begin{aligned}
	\frac{1}{\bigl|\cl\bigr|}&=\exp\bigg(\frac{1}{2n}(\left\langle C\ut, \ut\right\rangle+d_{n})\bigg)\\
	& =  \exp\bigg(\frac{1}{2n}\left(C^{11}U^{2}_{n}+C^{22}U^{2}_{n}+2C^{12}U^{2}_{n}+ d_n \right)\bigg)\\
	& \leq \exp\bigg(\frac{1}{2n}U^{2}_{n}\left(C^{11}+C^{22} +2C^{12}+4\int_{\R^{2}}(1\wedge\parallel\textbf{x}\parallel^{2})F(d\textbf{x})\right)\bigg)\\
	& \leq \exp\bigg(\frac{1}{2n}U^{2}_{n} \left(4\left(|C|_{\infty}+ \int_{\R^{2}}(1\wedge \parallel\textbf{x}\parallel^{2})F(d\textbf{x})\right)\right)\bigg)\\
	&\leq \exp\bigg(\frac{1}{2n}U^{2}_{n} \left(4M\right)\bigg).
	\end{aligned}
	\end{equation}
	
	The first inequality holds because 
	\[
	d_{n} = 2 \int \Big(1 - \cos( \langle \textbf{u}_{n}, \textbf{x}\rangle ) \Big)F(d\textbf{x}) \leq  2 \int \Big(1 \wedge |\langle \textbf{u}_{n}, \textbf{x}\rangle |^{2} \Big)F(d\textbf{x}) \leq 4 U_{n}^{2}\int (1 \wedge \|\textbf{x} \|^{2}) F(d\textbf{x}),
	\]
	using the Cauchy-Schwarz inequality for $ | \langle \textbf{u}_{n}, \textbf{x}\rangle|^{2}\leq \|\textbf{u}_{n}\|^{2}\|\textbf{x}\|^{2}$ and the fact that $ U_{n}\geq 1 $. The last inequality in (\ref{boch}) derives from Assumption\ref{as1} and the fact that we always have $ \int_{\R^{2}} (1 \wedge \|\textbf{x} \|^{2} ) F(d\textbf{x}) < \infty$. 
	Next, the form of $U_{n}$ (\ref{you}) implies that
	\begin{equation}\label{car}
	\frac{1}{\vert\cl\vert}\leq
	\begin{cases}
	\Gamma \quad & \mbox{if $r\leq 1$}\\
	4n^\frac{r-1}{2} \quad & \mbox{if $r>1$},
	\end{cases}
	\end{equation}
	where $\Gamma=e^{2M}$.
	Let us now argue that
	\[
	\frac{\f}{\fit}=\frac{\tilde{V}_{n}+\clt}{V_{n}+\cl}=\frac{\phi_{n}(\tilde{u}_{n})\left(1+\frac{\tilde{V}_{n}}{\clt}\right)}{\phi_{n}(u_{n})\left(1+\frac{V_{n}}{\cl}\right)}\neq 0, \infty 
	\]
	as soon as $n\geq n_{0}=(2\Gamma)^{\frac{4}{(2-r)\wedge r}}$ and by (\ref{car}) on the set $\left\{|V_{n}|\leq \frac{1}{n^{r/4}}\right\}$ and $\left\{|\tilde{V}_{n}|\leq \frac{1}{n^{r/4}}\right\}$.
	Therefore, $\left|\frac{V_{n}}{\phi_{n}}\right|\leq \frac{1}{2}$. Accordingly, for the stochastic error on the events $\left\{|V_{n}|\leq \frac{1}{n^{r/4}}\right\}$ and $\left\{|\tilde{V}_{n}|\leq \frac{1}{n^{r/4}}\right\}$ we obtain for some deterministic constant $ A $:
	\begin{equation}
	\begin{aligned}
	|H_{n}|&\leq\frac{n}{2U^{2}_{n}}\Bigg|\log \left|\frac{\f}{\fit}\right|-\log\left|\frac{\clt}{\cl}\right|\Bigg|\\
	%&=\frac{n}{2U^{2}_{n}}\left(\log \left|\frac{\f\cl}{\fit\clt}\right|\right)\\
	&=\frac{n}{2U^{2}_{n}}\Bigg|\log \left|\frac{\f}{\clt}\right|-\log \left| \frac{\fit}{\cl}\right|\Bigg|\\
	&=\frac{n}{2U^{2}_{n}}\Bigg|\log \left|1+\frac{\f-\clt}{\clt}\right|-\log \left| 1 +\frac{\fit-\cl}{\cl}\right|\Bigg|\\
	&=\frac{n}{2U^{2}_{n}}\Bigg|\log \left|1+\frac{\tilde{V}_{n}}{\clt}\right|-\log \left| 1 +\frac{V_{n}}{\cl}\right|\Bigg|\\
	&\leq  \frac{An}{2U^{2}_{n}}\Bigg|\left|\frac{\tilde{V}_{n}}{\clt}\right|-\left|\frac{V_{n}}{\cl}\right|\Bigg|.
	\end{aligned}
	\end{equation}
	In the last inequality, we use the linearized stochastic errors for  $ \log \left|1+\frac{\tilde{V}_{n}}{\clt}\right| \thickapprox  \left|\frac{\tilde{V}_{n}}{\clt}\right|$ because of the fact that $\frac{\tilde{V_{n}}}{\phi_{n}(\tilde{u_{n}})}$ and  $\frac{V_{n}}{\phi_{n}(u_{n})}$ are small enough. So there is a positive constant $ A $ such that  $ \log \left|1+\frac{\tilde{V}_{n}}{\clt}\right|\leq A\left|\frac{\tilde{V}_{n}}{\clt}\right| $. Therefore,
	\begin{equation}\label{sto}
	|H_{n}|\leq \frac{An}{U^{2}_{n}}\max\left(\left|\frac{\tilde{V}_{n}}{\clt}\right|,\left|\frac{V_{n}}{\cl}\right|\right).
	\end{equation}
	Henceforth, for $n\geq n_{0}$, and for some constant $A>0$, we have
	\begin{equation*}
	\begin{aligned}
	\E\left[|H_{n}|\mathds{1}_{\left\{|V_{n}|\leq \frac{1}{n^{r/4}}\right\}}\right]&\leq\E\left[\frac{An}{U^{2}_{n}}\max\left(\left|\frac{\tilde{V_{n}}}{\clt}\right|,\left|\frac{V_{n}}{\cl}\right|\right)\mathds{1}_{\left\{|V_{n}|\leq \frac{1}{n^{r/4}}\right\}}\right]\\
	&\leq \frac{An}{U^{2}_{n}}\left|\frac{1}{\cl}\right|\E\left[|V_{n}|\mathds{1}_{\left\{|V_{n}|\leq \frac{1}{n^{r/4}}\right\}}\right]\\
	%&=\frac{Bn}{2U^{2}_{n}}\left|\frac{1}{\phi_{n}(u_{n})}\right|\E\left(|V_{n}|\mathds{1}_{\left\{|V_{n}|\leq \frac{1}{n^{r/4}}\right\}}\right)^{2\frac{1}{2}}\\
	&\leq \frac{An}{U^{2}_{n}} \left|\frac{1}{\cl}\right|\E\left(|V_{n}|^{2}\right)^{\frac{1}{2}}\\
	&\leq \frac{An}{U^{2}_{n}} \left|\frac{1}{\cl}\right| \frac{1}{\sqrt{n}}.
	\end{aligned}
	\end{equation*}
	The third inequality holds because we applied the Cauchy-Schwarz inequality and by Lemma \ref{unvar}. To sum up, by (\ref{car}) we get that
	\begin{equation}\label{exsto}
	\E\left[|H_{n}|\mathds{1}_{\left\{|V_{n}|\leq \frac{1}{n^{r/4}}\right\}}\right]\leq
	\begin{cases}
	\frac{A\Gamma}{\sqrt{n}} \quad & \mbox{if $r\leq 1$}\\
	\frac{4AM}{(r-1)n^{\frac{2-r}{2}}\log n }\quad & \mbox{if $r>1$}
	\end{cases}
	\end{equation}
	as required.

	\end{proof}
\begin{remark}
	Here, we are interested in the events $\left\{|V_{n}|\leq \frac{1}{n^{r/4}}\right\}$ and $\left\{|\tilde{V}_{n}|\leq \frac{1}{n^{r/4}}\right\}$ because the probabilities of the events $\left\{|V_{n}|>\frac{1}{n^{r/4}}\right\}$ and $\left\{|\tilde{V}_{n}|> \frac{1}{n^{r/4}}\right\}$ are negligible. Indeed, applying the Chebyshev inequality and by Lemma \ref{unvar} we get
	\[
	\Pm\left(|V_{n}|>\frac{1}{n^{r/4}}\right)\leq n^{r/2}\E\left(|V_{n}|^{2}\right)\leq n^{r/2}\frac{1}{n}=n^{\frac{r-2}{2}},
	\]
	which tends towards zero as $n\to \infty$. Likewise, the probability of the event $\left\{|\tilde{V}_{n}|> \frac{1}{n^{r/4}}\right\}$ tends towards zero as $ n\to \infty $.
\end{remark}
Until now we bound from above the deterministic and stochastic errors. We are now ready to prove that the family $\frac{1}{w_{n}} |\widehat{C}_{n}^{12} -C^{12}|$ is tight in $ \class$ and thus establish an upper bound for our estimator.
\subsection* {End proof of \autoref{upbo}}
Applying the Markov inequality, we get for every $ \epsilon, L>0 $ 
\begin{equation}
\begin{split}
\Pm\left( \frac{1}{w_{n}}|\widehat{C}_{n}^{12}\left(U_{n}\right)-C^{12}|\geq \epsilon\right)& =  \Pm\left(\frac{1}{w_{n}}|\widehat{C}_{n}^{12}\left(U_{n}\right)-C^{12}| > \epsilon, |V_{n}|\leq \frac{1}{n^{r/4}}\right) +\Pm\left(|V_{n}|>\frac{1}{n^{r/4}}\right)\\
&\leq \frac{1}{\epsilon} \E \left[ \frac{1}{w_{n}}|\widehat{C}_{n}^{12}\left(U_{n}\right)-C^{12}| \mathds{1}_{\{|V_{n}|\leq \frac{1}{n^{r/4}}\}} \right] + n^{\frac{r-2}{2}}\\
&\leq \frac{1}{\epsilon}\E\left[\frac{1}{w_{n}}|H_{n}|\mathds{1}_{\{|V_{n}|\leq \frac{1}{n^{r/4}}\}}\right]+ \frac{1}{\epsilon}\frac{1}{w_{n}}|D_{n}| + n^{\frac{r-2}{2}}.
\end{split}
\end{equation}
Further applying Lemmas \ref{stoer} and \ref{deterr}, we deduce that, as $ n\to \infty $
\begin{equation}\label{prob}
\Pm\left( \frac{1}{w_{n}}|\widehat{C}^{12}_{n}\left(U_{n}\right)-C^{12}| \geq \epsilon\right)\leq 
\begin{cases}
 \frac{\Gamma}{\epsilon}\quad & \mbox{if $r\leq 1$}\\
\frac{1}{\epsilon} \frac{2^{\frac{r-2}{2}}M^{r/2}}{(r-1)^{\frac{r-2}{2}}}\quad & \mbox{if $r>1$}
\end{cases}
\end{equation}
for (\ref{prob}) smaller than $ L$, proves that the family $\frac{1}{w_{n}}|\widehat{C}_{n}^{12}(U_{n})-C^{12}|$ is tight in $\textbf{X} \in \class $. The proof is complete.

%%%%%%%%%%%%%%%%%%%%%%%%%%%%%%%%%%%%%%%%%%%%%%%%%%%%%%%%%%

\section{Lower Bound}\label{sec:lowerbound}

In nonparametric statistics it is common to use a minimax approach in order to prove optimal estimators. In the previous section, we proved \autoref{upbo}, which gave us an upper bound for the estimation of co-integrated volatility using a spectral approach and establishing the rates (\ref{rates}) on the class $ \class $. \par
In this section, we want to prove \autoref{lobo}. The existence of a lower bound on the class $\class$ constitutes the exact minimax rates for the estimation of co-integrated volatility. Indeed, we have something more for the lower bound, namely that any estimator on a general class of It\^o semimartingales satisfying Definition \ref{class} achieves a lower bound with rates (\ref{rates}).  So far, we do not know whether the spectral approach for the upper bound yields the same optimal rate on the larger class of It\^o semimartingale.\par
We refer to Chapter 2 in \cite{tsybakov2009introduction} for the techniques to prove the lower bounds. We establish the lower bound following the argumentation in line with a two-hypothesis test. Next, we introduce a distance between probability measures that will be useful for the lower bound.
\begin{defn} The \textbf{total variation distance} between $\Pm_{0},\Pm_{1}$ is defined as follows:
	\[
	\tv(\Pm_{0},\Pm_{1}):= \sup\left | \int (p_{0}-p_{1})\nu(dx)\right |.
	\]
	where $p_{0} = d\Pm_{0}/d\nu$, $p_{1} = d\Pm_{1}/d\nu$  and $\nu= \Pm_{0}+ \Pm_{1}$ a $\sigma$-finite measure.
\end{defn}
To sum up, in order to prove a lower bound on the minimax probability of error for hypotheses we use the theorem [2.2] in \cite{tsybakov2009introduction}. The lower bound is obtained when the following two properties are satisfied. First, we choose the appropriate parameters for the co-integrated volatility to be close enough but distinguished. Second, we bound from below the total variation distance between the two densities probabilities of our parameters.\par
Let us illustrate the above procedure giving a trivial lemma and proving that the above arguments are adequate, so as to obtain the lower bound corresponding to our family of estimators $\class$. The interested reader may refer to \cite{lehmann2006testing} who explore a lot of examples for hypothesis testing and distances between Gaussian random variables. 

\begin{lemma}\label{twodim}(No jumps for a two-dimensional L\'evy process).
	Assume $\textbf{X}$ belongs to the class $\class$ with L\'evy-Khintchine triplet $(0,\Sigma\Sigma^{\top}, 0)$. Then there are constants $ A, K $ such that
	\[
\liminf_{n\to \infty}	\inf_{\estc}\sup_{\cov \in \class}\Pm[\text{d}(\estc, \cov)> Aw_{n}]\geq K,
	\]
	where $\estc$  is any estimator for the co- integrated volatility withing the class $ \class $, $d$ is the euclidean distance on $\R^{2}$ and $w_{n}=\frac{1}{n}$.
\end{lemma}

\begin{proof}
	Consider $\textbf{X}$ and $\textbf{Y}$ belongs to $\class$. Also, we assume that no jumps are occurred so the L\'evy Khintchine triplets for each process will satisfy $(0, \Sigma_{\textbf{X}}\Sigma_{\textbf{X}}^{\top},0)$ and $(0,\Sigma_{\textbf{Y}}\Sigma_{\textbf{Y}}^{\top},0)$ respectively.  
	As a result, $ \textbf{X}$ will evolve as follows:
	\[
	dX^{(1)}(t)=\sigma^{(1)}_{t}dW^{(1)}_{t}
	\] 
	\[
	dX^{(2)}(t)=\sigma^{(2)}_{t}dW^{(2)}_{t}
	\]
	similarly for $ \textbf{Y} $. We know that the It\^o integral $ dX^{(1)}(t)=\sigma^{(1)}_{t}dW^{(1)}_{t} $ is normally distributed with mean $ 0 $  and its variance is given by It\^o isometry which translates to 
	\[
	X^{(1)}\thicksim \mathcal{N}\left(0, \int_{0}^{1}(\sigma^{(1)}_{t})^{2}dt\right).
	\]
	Therefore, $\textbf{X}$ follows the parametric model 
	\[
	\textbf{X}=\left(\begin{smallmatrix}X_{(1)}\\X_{(2)}  \end{smallmatrix}\right)\thicksim\mathcal{N}\bigg( \left(\begin{smallmatrix}0\\0  \end{smallmatrix}\right), \Sigma_{\textbf{X}}\Sigma_{\textbf{X}}^{\top}\bigg)
	\]
	similarly for  $ \textbf{Y} $.\par
	We will prove the lower bound using the two-hypothesis test, as mentioned in the beginning of this section. 
	We observe that
	\[
	\class \supset \mathcal{B}_{M}
	\]
	where $\mathcal{B}_{M}$ is the class of all Brownian motions where the covariance matrix is bounded component-wise by a constant $ M $. As a consequence, 
	\[
	\sup_{\cov \in \mathcal{B}_{M}} \Pm[d(\estc,\cov)]\leq \sup_{\cov \in \class} \Pm[d(\cov,\estc)].
	\]
	This is enough to prove a lower bound for the rate $w_{n}=\frac{1}{n}$ for the class of all Brownian motions. \par
	The two-hypothesis test is the following
	\begin{center}
	
	$\Pm_{\textbf{X}}= \mathcal{N}\left(0, \Sigma_{\textbf{X}}\Sigma_{\textbf{X}}^{\top}\right) $ 
	vs
	$
	\Pm_{\textbf{Y}}=\mathcal{N}\left(0, \Sigma_{\textbf{Y}}\Sigma_{\textbf{Y}}^{\top}\right),
	$
\end{center}
	where the covariance matrices are given by $ \Sigma_{\textbf{X}}\Sigma_{\textbf{X}}^{\top}= \left(\begin{smallmatrix}2 & 1\\ 1& 1 \end{smallmatrix}\right)$ and $ \Sigma_{\textbf{Y}}\Sigma_{\textbf{Y}}^{\top}= \left(\begin{smallmatrix}2 & 1+\frac{1}{n}\\ 1+\frac{1}{n}& 1 \end{smallmatrix}\right)$. Intuitively, we perturb the off-diagonal elements, namely the covariance, by the rate we want to achieve. Following the argumentation of the two-hypothesis test, it is sufficient to prove that the total variation distance is bounded. To do so, we use the Pinsker inequality. By Pinsker inequality we have that 
	\[
	\tv(\Pm_{\textbf{Y}},\Pm_{\textbf{X}})\leq \sqrt{KL(\Pm_{\textbf{Y}},\Pm_{\textbf{X}})/2},
	\]
	where $KL(\Pm_{\textbf{Y}},\Pm_{\textbf{X}})$ is the Kullback-Leibler divergence. Next, we show that the Kullback-Leibler distance is bounded. We define the Kullback-Leibler divergence between two multivariate normal distributions. Here, we denote by $ \Sigma_{1}=\Sigma_{\textbf{X}}\Sigma_{\textbf{X}}^{\top} $ and $\Sigma_{2}=\Sigma_{\textbf{Y}}\Sigma_{\textbf{Y}}^{\top} $. Therefore,
	\[
	KL(\Pm_{\textbf{Y}},\Pm_{\textbf{X}})=\frac{1}{2}\left(\log \frac{|\Sigma_{1}|}{|\Sigma_{2}|}-2 +\text{tr}(\Sigma_{1}^{-1}\Sigma_{2})\right),
	\]
	where $ | \cdot | $ denotes the determinant of a matrix. Calculating the appropriate quantities, we obtain  
	\[
	|\Sigma_{1}|= \left |\begin{smallmatrix}2 & 1\\1 &1 \end{smallmatrix}\right|=1,
	\]
	
	\[
	|\Sigma_{2}|= \left | \begin{smallmatrix} 2& 1+\frac{1}{n}\\1+
	\frac{1}{n} &1 \end{smallmatrix}\right |=2-\left(1+\frac{1}{n}\right)^{2},
	\]
	\[
	\text{tr} (\Sigma_{1}^{-1}\Sigma_{2})= 2-\frac{2}{n}.
	\]
	Therefore,
	\[
	KL(\Pm_{\textbf{X}}, \Pm_{\textbf{Y}})=\frac{1}{2}\left(\log\left(\frac{1}{2- \left(1 + \frac{1}{n}\right)^{2}}\right)-2+2-\frac{2}{n}\right).
	\]
	Consequently, we obtain that the right hand side tends to zero as $ n\to \infty $. By Pinsker inequality, the total variation distance tends to zero. Upon using the minimax probability of error is bounded from below by $ 1/2 $ and the claim follows.
\end{proof}
To prove \autoref{lobo} we need to construct the two-hypothesis test in order to bound from below the minimax probability error as we described previous in \Cref{twodim}.

%%%%%%%%%%%%%%%%%%%%%%%%%%%%%%%%%%%%%%%%%%%%%%%%%%%%%%%%%%%%%%%%%%%%%%%%5
\subsection{Two-hypothesis test}\label{test}
We let $\textbf{X}$, $\textbf{Y}$ be two-dimensional L\'evy processes with respective triplets $(0,\Sigma_{\textbf{X}}\Sigma_{\textbf{X}}^{\top}, F_{n} )$,  $(0,\Sigma_{\textbf{Y}}\Sigma_{\textbf{Y}}^{\top}, G_{n} )$, where $F_{n}$, $G_{n}$ are L\'evy measures in $\R^{2}$ satisfying
\begin{equation}\label{small}
\int_{\R^{2}} \big( 1 \wedge |x_{1}x_{2}|^{r/2} \big)F_{n}(dx_{1},dx_{2})\leq M, \qquad \int_{\R^{2}}(1\wedge |x_{1}x_{2}|^{r/2})G_{n}(dx_{1},dx_{2})\leq M,
\end{equation}
where $\textbf{x}=(x_{1}, x_{2})$ is a vector in $\R^{2}$ representing the size of small jumps for each process and $M$ is a constant (below $ M $ changes from line to line and may depend on r, but all constants are denoted as $M$).
We set $\Sigma_{\textbf{X}}\Sigma_{\textbf{X}}^{\top}=\left(\begin{smallmatrix}2& 1\\1 &1 \end{smallmatrix}\right)$ and $\Sigma_{\textbf{Y}}\Sigma_{\textbf{Y}}^{\top}=\left(\begin{smallmatrix}2& 1+2w_{n}\\1 +2w_{n}&1 \end{smallmatrix}\right)$ to be the parameters of our two-hypothesis test. Under this setting, we perturb the off-diagonal elements with the rate with which we want to achieve the lower bound. The quantity which we want to recover is the co-integrated volatility, so we need the off-diagonal elements. We use these forms of matrices in order for the Gaussian part to be non-degenerated, namely the eigenvalues of the matrices to be positive. As we discussed in the beginning of this section, it is sufficient to construct two sequences $\textbf{X}_{n}$, $\textbf{Y}_{n}$ which belong to the class $\class$, with the following two properties:\\
\begin{property}\label{prop1}
	The two parameters, namely the two covariance matrices are close enough but distinguished. 
\end{property}
Note that for this property the object of our study is the distance between matrices. In this case we consider as a distance the Frobenius norm, and everything still holds. By construction and Frobenius norm
\[
\begin{split}
||\Sigma_{\textbf{Y}}\Sigma_{\textbf{Y}}^{\top}-\Sigma_{\textbf{X}}\Sigma_{\textbf{X}}^{\top}||_{F}&=\sqrt{\text{tr}(\Sigma_{\textbf{Y}}\Sigma_{\textbf{Y}}^{\top}-\Sigma_{\textbf{X}}\Sigma_{\textbf{X}}^{\top})(\Sigma_{\textbf{Y}}\Sigma_{\textbf{Y}}^{\top}-\Sigma_{\textbf{X}}\Sigma_{\textbf{X}}^{\top})^{\top}}\\
&=\sqrt{8}w_{n},
\end{split}
\]
which means that the parameters are close enough but distinguished. 
\begin{property}\label{prop2}
	The total variation distance between $\Pm_{\textbf{X}}$ and $\Pm_{\textbf{Y}}$ tends towards zero.
\end{property}
As far as the second property is concerned, the total variation distance tends towards zero is not trivial. In fact, achieving the second property is quite demanding and we prove several lemmas to conclude this property.

%%%%%%%%%%%%%%%%%%%%%%%%%%%%%%%%%%%%%%%%%%%%%%%%%%%%%%%%%%%%%%%%%%%%%%%%%
\subsection{Construction of the co-jump measure in $\R^{2}$}
First, we have to construct a measure to satisfy property  (\ref{small}). Before we proceed with the technical part of this construction, let us highlight the idea behind it.\par 
Note that we are studying a two-dimensional L\'evy process, so it is reasonable to include the possibility of dependence between the two jump components, more specifically the common jumps, i.e. the co-jumps. \par 
Observe here that co-jumps are one-dimensional objects. Co-jumps are the jumps on the diagonal, due to the fact that the two processes jump at the same time with the same jump size. Mathematically speaking, this can be formalized as follows:
\begin{defn}(Co-jump measure)
	Let $\textbf{X}=\left(X^{(1)}, X^{(2)}\right)$  be a L\'evy processes, with $\Delta X^{(j)}_{t}\neq 0$ for $j=1,2$. Here, $\Delta X^{(j)}_{t}=X_{t}^{(j)}-X^{(j)}_{t^{-}}$ denotes the possible jump at time $t$. The measure on $\R^{2}$ is defined by:
	\[
	F_{n}(B)= \E \left[\# \left\{ t \in [0,1]: (\Delta X^{(1)}_{t},\Delta X^{(2)}_{t}) \in B \right\}\right]=\E \left [\mu^{X^{(1)}X^{(2)}} (\omega; t, B)\right],
	\]
	where $B=\left\{ (x_{1},x_{2})\in \R^{2} : x_{1}=x_{2} \right\}$. This is called the L\'evy measure in $\R^{2}$  of \text{co-jumps} for $\textbf{X}$. $F_{n}(B)$ is the expected number of joint jumps, per unite of time, whose size falls into $B$, and $\mu$ is the Poisson random measure of co-jumps where, $\mu^{X}(\omega;t,B)=\sum_{s\leq t} \mathds{1}_{B}(\Delta X^{(1)}_{t},\Delta X^{(2)}_{t})$.
\end{defn}
Because the jump dynamics of the co-jump measure is dictated by its density, say $f_{n}$, we can write the measure of the co-jumps as following, for $A\subset B \subset \R^{2}$ 
\begin{equation}\label{cojump1}
F_{n}(A):= \int f_{n}(x) \mathds{1}_{A}(x,x) dx:=\int f_{n}(x)\mathds{1}_{\tilde{A}}(x)dx,
\end{equation}
where $ \tilde{A} =\{x:(x,x)\in A\}$.\par 
The support of the co-jump measure is on $ \R $ but the co-jump measure lives on $ \R^{2} $. We focus on the case of co-jumps, i.e., when $X^{(1)}$ and $X^{(2)}$ jump at the same time with the same jump size. We are interested in the jumps on the diagonal.\par 
Further, we do not integrate with respect to the Lebesgue measure, since it is equal to zero on the diagonal. In this case we integrate with respect to a measure that is not absolutely continuous with the Lebesgue measure, which we call \textit{co-jump measure}. We assume that $ F_{n}, G_{n} $ have densities $ f_{n}  $ and $g_{n} $ respectively. By (\ref{cojump1}) we want to show that 
\begin{equation}\label{onedim}
\int_{\R^{2}}g(x_{1},x_{2})dF_{n}(x_{1},x_{2})=\int_{\R} g(x, x)f_{n}(x)dx
\end{equation}
without being equal to zero. Being interested in the set of co-jumps, we pass from two dimensions to one dimension. Co-jumps are the concept of total dependency between the small jump components. Indeed, we use the argument of dependency in order to reduce dimensionality. In order to prove this argument, we need the following lemma. 
\begin{lemma}\label{diagon}
	Let $g: \R^{2}\to \R$ be a measurable function and $F_{n}$ be the co-jump measure on $\R^{2}$. Then 
	\[
	\int_{\R^{2}} g(x_{1}, x_{2})d F_{n}(x_{1}, x_{2})=\int_{B}  g(x_{1}, x_{2})d F_{n}(x_{1}, x_{2})= \int f_{n}(x) g(x,x)dx
	\]
	where $F_{n}(A)=\int f_{n}(x)\mathds{1}_{A} (x,x) dx$ is the measure of co-jumps, $f_{n}$ the density function of the co-jump measure $F_{n}$, $A\subset B $, and $B=\left\{ (x_{1}, x_{2})\in \R^{2}: x_1=x_2\right\}$.
\end{lemma}
\begin{proof}
	First we use the step functions to prove the lemma.This extends by linearity and by taking limits for all measurable functions $g$. Indeed, we only need to show the lemma for the case of step functions. Let $g(x_{1}, x_{2})=\sum_{k = 1}^{m} a_{k}\mathds{1}_{A_{k}}(x_1, x_2)$, where $A_{k}\subset A$ and $\cup_{k=1}^{m}A_{k}=A$. Therefore,
	\begin{equation}
	\begin{aligned}
	\int g(x_{1}, x_{2})d F_{n}(x_{1}, x_{2})&=\int \sum_{k=1}^{m}a_{k}\mathds{1}_{A_{k}}(x_{1},x_{2})dF_{n}(x_{1},x_{2})\\
	&=\sum_{k=1}^{m}a_{k}\int \mathds{1}_{A_{k}}(x_{1},x_{2})dF_{n}(x_{1},x_{2})\\
	&=\sum_{k=1}^{m}a_{k}\int_{A_{k}} dF_{n}(x_{1},x_{2})\\
	%&=\sum_{k=1}^{m}a_{k}F_{n}(A_{k})\\
	&=\sum_{k=1}^{m}a_{k}\int \mathds{1}_{A_{k}}(x,x)f_{n}(x)dx\\
	&=\int \sum_{k=1}^{m}a_{k}\mathds{1}_{A_{k}}(x,x)f_{n}(x)dx\\
	&=\int g(x,x)f_{n}(x)dx,
	\end{aligned}
	\end{equation}
	and the claim follows.
\end{proof}
Furthermore, we need to find a measure whose mass is bounded away from the origin but may explode around 0 and integrates $\| x\| ^{2}$. In order to construct the co-jump measure with the above properties, we need to find an appropriate density function for the co-jumps measure $F_{n}(A)$ so as to satisfy the following condition for $r\in(1,2)$ and $ \textbf{x}=(x,x) $:
\[
\int_{A} \left(1\wedge |x_{1}x_{2}|^{r/2}\right)F_{n}(dx_{1},dx_{2})=\int \left(1\wedge|x|^{r}\right)f_{n}(x)\mathds{1}_{A}(x,x)dx< \infty.
\]
Indeed, the following lemma implies condition (\ref{small}) by choosing properly the density function of the co-jumps.
\begin{prop}\label{propo}
	Let $ w_{n} $ be defined by (\ref{rates}) and $r\in(0,2)$. Assume the even functions $h_{n}: \R^{2} \to \R$ such that $h_{n}(\textbf{u})=\tilde{h}_{n}(U)\cdot \tilde{h}_{n}(U)$, where $\textbf{u}=(U,U)$,
	\[
	\tilde{h}_{n}(U)=
	\begin{cases}
	\sqrt{w_{n}}&\quad \mbox{$U\leq U_{n}$}\\
	\sqrt{w_{n}}e^{-(U-U_{n})^{3}}& \quad \mbox{$U>U_{n}$},
	\end{cases}
	\]
	and $U_{n}= 2 w_{n}^{1/(r-2)} $. Then, for any $A \in \mathcal{B}(\R^{2}) $
	\begin{equation}\label{small2}
	\int (1\wedge |x|^{r})\mathds{1}_{A}(x,x)F_{n}(dx) < \infty,
	\end{equation}
	where $F_{n}(A)=\int_{\R}\frac{|H_{n}(x)|}{x^{2}}\mathds{1}_{A}(x,x)dx$ and $ H_{n} $ is the Fourier transform of $ h_{n} $.
\end{prop}
\begin{proof}
	The mathematical tool used for the formation of the density function is the Fourier transform. Intuitively, we use the function $ h_{n} $ as a constant inside a fixed interval and which decays exponentially outside this interval. Also, notice that in the exponential we used the power of 3 because we need to differentiate two times, as we shall see later.\par 
	Notice that $h_{n}$ has a range on $\R$, which is why we use the Fourier transform on $\R$. The pair of Fourier transform takes the following form:
	\[
	H_{n}(x)= \left(\mathcal{F}h_{n}\right)(x) =\int e^{iUx}h_{n}(U)dU
	\]
	\[
	h_{n}(U)=\left( \mathcal{F}^{-1}H_{n}\right)(U) =\frac{1}{2\pi}\int e^{-i Ux}H_{n}(x)dx
	\]
	and the respective first derivatives will have the form
	\[
	\partial_{1} H_{n}=xH_{n}(x)
	\]
	\[
	\partial_{1}h_{n}=i\mathcal{F}^{-1}\partial_{1}H_{n}(x).
	\]
	For a thorough analysis of the Fourier transform the interested reader should refer to \cite{bracewell1986fourier}. \par 
	In the next step, the pair of Fourier transform will provide us with a proper and well-defined density function for the co-jump measure. First, we note that the $\mathbb{L}^{2}$-norm of $h_{n}$ is bounded. Indeed,
	
	\begin{equation} \label{hnorm}
	\begin{aligned}
	\|h_{n}\|_{L^{2}}=&\| \tilde{h}_{n}\|^{2}_{L^{2}}=\int |\tilde{h}_{n}(U)|^{2}dU\\
	&= w_{n}\left(\int_{U\leq U_{n}} dU + \int_{U> U_{n} }  e^{-2(U-U_{n})^{3} }dU\right)\\
	&= w_{n}\left(U_{n} + \int_{U_{n}}^{\infty} e^{-2(U-U_{n})^{3} }dU\right)\\
	&=  w_{n}\left(U_{n} + \int_{0}^{\infty} e^{-2K^{3} }dK\right)\\
	& \leq C w_{n} U_{n}=C w_{n}^{\frac{r-1}{r-2}}.
	\end{aligned}
	\end{equation}
	In the last inequality we used the fact that $\int_{0}^{\infty} e^{-2K^{3}}dK$ is bounded by a constant $C$. In addition, $h_{n}$ is an $\mathbb{L}^{2}$-function. Applying the Plancherel theorem we deduce that 
	\begin{equation} \label{Hnorm}
	\|H_{n}\|_{\mathbb{L}^{2}}=\|h_{n}\|_{\mathbb{L}^{2}}\leq C w_{n}^{\frac{r-1}{r-2}}.
	\end{equation} 
	Similarly, we get a bound for the first derivative of $H_{n}$
	\begin{equation}\label{hdernor}
	\| \partial_{1}H_{n}\|\leq \|\partial_{1}h_{n}\|_{\mathbb{L}^{2}}\leq Cw_{n}.
	\end{equation}
	Moreover, $\|H_{n}\|_{\mathbb{L}^{1}}$ is also bounded
	\begin{equation}\label{inth}
	\begin{split}
	\int |H_{n}(x)| dx&=\int \frac{1}{\sqrt{1+x^{2}}}\sqrt{1+x^{2}}|H_{n}(x)|dx\\
	&\leq \left(\int \frac{1}{1+x^{2}}dx\right)^{1/2}\left(\int H^{2}_{n}(x)(1+x^{2}) dx \right)^{1/2}\\
	& \leq C \left( H^{2}_{n}(x)+x^{2}H^{2}_{n}(x) dx \right)^{1/2}\\
	&\leq C\left(\|H_{n}\|_{\mathbb{L}^{2}}+\|\partial_{1}H_{n}\|_{\mathbb{L}^{2}} \right) \\
	&\leq C(1+ w_{n}^{\frac{r-1}{r-2}}).
	\end{split}
	\end{equation}
	We get the first inequality because of the Cauchy-Schwarz inequality. By means of (\ref{Hnorm}) and (\ref{hdernor}) the $\mathbb{L}^{1}$-norm of $H_{n}$ is bounded.\par 
	At this point we are ready to define the co-jumps measures $F_{n}(A)$ and $G_{n}(A)$ in terms of the Fourier transform $H_{n}(x)$. 
	\begin{equation}\label{cojump}
	F_{n}(A)=\int _{\R}\frac{|H_{n}(x)|}{x^{2}}\mathds{1}_{A}(x,x)dx \qquad G_{n}(A)=\int _{\R}\left(F_{n}(A)+\frac{H_{n}(x)}{x^{2}}\right)\mathds{1}_{A}(x,x)dx.
	\end{equation}
	for any $ A\in\mathcal{B}(\R^{2}) $. \par 
	These measures satisfy the basic properties of L\'evy measures. They are non-negative, integrate $x^{2}$, and may explode around zero since $H_{n}(0) \to \infty$. \par 
	It remains to prove (\ref{small}) under this argumentation. Based on the above construction and $ A=\{(x_1,x_2)\in \R^{2}:x_1=x_2=x\} $, (\ref{small}) transforms into:
	\begin{equation}\label{cojmeasure}
	\int (1\wedge |x|^{r})\frac{|H_{n}(x)|}{|x|^{2}}dx,
	\end{equation}
so we need to show that (\ref{cojmeasure}) is finite. Next we show how to bound from above $|H_{n}(x)|$. 
	\begin{equation}
	\begin{split}
	|H_{n}(x)|&\leq \int |e^{iUx}h_{n}(U)|dU \leq \int |h_{n}(U)\cos(Ux)|dU+i\int|h_{n}(U)\sin(Ux)|dU\\
	&=|h_{n}(U)\cos(Ux)|dU.
	\end{split}
	\end{equation}
	In the first line the second term is equal to zero since it is the integral of the product of an even and an odd function.\\ 
	\begin{equation}
	\begin{split}
	|H_{n}(x)|&\leq 2w_{n}\int_{0}^{U_{n}}|\cos(Ux)|dU+2w_{n}\int_{U_{n}}^{\infty}|e^{-2(U-U_{n})^{2}}\cos(Ux)|dU\\
	&\leq 2w_{n}\left(\frac{|\sin(U_{n}x)|}{|x|}+\int_{0}^{\infty} e^{-2K^{3}}\cos( (K+U_{n})x)dK\right)\\
	&\leq C w_{n}\left(\frac{|\sin(U_{n}x)|}{|x|} +1\right).     
	\end{split}
	\end{equation}
	Note that in the second inequality the integral is always bounded from above by a constant $C$. \par 
	On the sets $\left\{|x|\leq \frac{1}{U_{n}}\right\} $, $\left\{\frac{1}{U_{n}}<|x|\leq 1 \right\}$, $\left\{|x|>1 \right\}$ we deduce that
	\begin{enumerate}
		\item 
		$|x|\leq \frac{1}{U_{n}}\Rightarrow |U_{n}x|\leq 1 \Rightarrow |\sin(U_{n}x)| \leq U_{n}x\Rightarrow \frac{|\sin(U_{n}x)|}{|x|}+1\leq U_{n}$.
		\item 
		$  \frac{1}{U_{n}}<|x|\leq 1\Rightarrow \frac{|\sin(U_{n}x)|}{|x|}+1=\frac{|\sin(U_{n}x)|+|x|}{|x|}\leq \frac{2}{|x|} $.
		\item
		$|x|>1 \Rightarrow \frac{|\sin(U_{n}x)|}{|x|}+1\leq 2$. 
	\end{enumerate}
	In turn, we get that
	\begin{equation}\label{Hnx}
	|H_{n}(x)|\leq C w_{n}\left( U_{n}\mathds{1}\left(|x|\leq \frac{1}{U_{n}}\right)+\frac{1}{|x|}\mathds{1}\left(\frac{1}{U_{n}}<|x| \leq 1\right)+\mathds{1}(|x|>1)\right).
	\end{equation}
	By splitting the integration domain into the sets $\left\{|x|\leq \frac{1}{U_{n}}\right\} $, $\left\{\frac{1}{U_{n}}<|x|\leq 1 \right\}$, $\left\{|x|>1 \right\}$ and recalling that $r\in(1,2)$, condition (\ref{small}) will take the form:
	\begin{equation*}
	\begin{aligned}
	&\int (1 \wedge |x|^{r})\frac{|H_{n}(x)|}{|x|^{2}} dx\\
	&\leq C w_{n}\int \frac{ 1 \wedge |x|^{r}}{|x|^{2}}\left(U_{n}\mathds{1}\left(|x|\leq \frac{1}{U_{n}}\right)+\frac{1}{|x|}\mathds{1}\left(\frac{1}{U_{n}}<|x| \leq 1\right)+\mathds{1}(|x|>1)\right)dx\\
	&\leq Cw_{n}U_{n}^{2-r}\leq C.
	\end{aligned}
	\end{equation*}
	In light of the form of $U_{n}$ the last inequality holds. Recall that $U_{n}= 2w_{n}^{1/r-2} $. Therefore, (\ref{small2}) is satisfied, which implies condition (\ref{small}), by which the proof is complete.
\end{proof}
Till now we constructed the co-jump measure, which satisfies (\ref{small}), and the covariance matrices for the hypothesis test. So the triplets for the hypothesis test are now defined. Next step, we study the characteristic functions of the two processes, which will be useful later on the proof of \Cref{prop2}.
%%%%%%%%%%%%%%%%%%%%%%%%%%%%%%%%%%%%%%%%%%%%%%%%%%%%%%%%%%%%%%%%%%%%%%%

\subsection {Characteristic functions of $\textbf{X}_{1/n}$ and $\textbf{Y}_{1/n}$}
At this point, we study the processes $ \textbf{X},\textbf{Y} $ for one observation at the moment $t=\frac{1}{n}$. We denote by $\psi_{n}(\textbf{u})$, $ \phi_{n}(\textbf{u})$ the characteristic functions of $\textbf{X}_{1/n}$, $ \textbf{Y}_{1/n} $ respectively, and by $ \eta_{n}(\textbf{u})  = \psi_{n}(\textbf{u}) - \phi_{n}(\textbf{u})$ their difference. The characteristic triplet for each process is
\[
\textbf{X}_{1/n}\sim \left( \textbf{0}, \Sigma_{\textbf{X}}\Sigma_{\textbf{X}}^{\top}, G_{n}(d\textbf{x}) \right)
\]
\[
\textbf{Y}_{1/n}\sim \left( \textbf{0}, \Sigma_{\textbf{Y}}\Sigma_{\textbf{Y}}^{\top}, F_{n}(d\textbf{x}) \right),
\]
where $ \Sigma_{\textbf{X}}\Sigma_{\textbf{X}}^{\top}=\left(\begin{smallmatrix}2 & 1\\1 &1 \end{smallmatrix}\right) $ and $\Sigma_{\textbf{Y}}\Sigma_{\textbf{Y} }^{\top}=\left(\begin{smallmatrix}2 & 1+2w_{n}\\1+2w_{n} &1 \end{smallmatrix}\right) $. Denote by $ C_{\textbf{X}}=\Sigma_{\textbf{X}}\Sigma_{\textbf{X}}^{\top}$ and  $ C_{\textbf{Y}}=\Sigma_{\textbf{Y}}\Sigma_{\textbf{Y}}^{\top}$. The characteristic functions will be defined as follows
\begin{equation}
\phi_{n}(\textbf{u})=\exp \left\{-\frac{1}{2n}\left( \langle C_{\textbf{Y}}\textbf{u, \textbf{u}}\rangle +2\tilde{\phi}_{n}(\textbf{u})\right) \right\}
\end{equation}
and
\begin{equation}
\psi_{n}(\textbf{u})=\exp \left\{-\frac{1}{2n}\left( \langle C_{\textbf{X}}\textbf{u, \textbf{u}}\rangle +2\tilde{\phi}_{n}(\textbf{u})+2\tilde{\psi}_{n}(\textbf{u})\right) \right\}.
\end{equation}
We denote by
\begin{equation}\label{phiti}
\tilde{\phi}_{n}(\textbf{u}) = \tilde{\phi}_{n}(U)=\int_{A}\big(1-\cos (Ux)\big) \frac{|H_{n}(x)|}{x^{2}}dx
\end{equation}
and
\begin{equation}\label{psiti}
\tilde{\psi}_{n}(\textbf{u}) = \tilde{\psi}_{n}(U)=\int_{A}\big(1-\cos (Ux)\big)\frac{H_{n}(x)}{x^{2}}dx
\end{equation}
because of the form of co-jump measure (\ref{cojump}) and the fact that $H_{n}$ is an even function, its Fourier transform will be a real function. Also, recall the Fourier transform of the co-jump measure has support on $\R $ and $A$ is a subset of the diagonal. Moreover, 
\[
\langle C_{\textbf{X}}\textbf{u, \textbf{u}}\rangle=5U^{2}\qquad \mbox{and} \qquad \langle C_{\textbf{Y}}\textbf{u, \textbf{u}}\rangle= 5U^{2}+4w_{n}U^{2}.
\]
Next we bound from above (\ref{phiti}) and (\ref{psiti}). First, observe that
\[
\tilde{\psi}''_{n}(U)=\int \cos(Ux)H_{n}(x)dx=h_{n}(U),
\]
since $H_{n}$ is an even function. We consider the following two cases: 
\begin{equation}\label{psitil}
\begin{split}
&|U|\leq U_{n} \Rightarrow \tilde{\psi}''_{n}(U)=w_{n}
\Rightarrow  \tilde{\psi}'_{n}(U) = w_{n} U \Rightarrow  \tilde{\psi}_{n}(U) = w_{n}\frac{U^{2}}{2} \\
&|U|\geq U_{n} \Rightarrow \tilde{\psi}'_{n}(U) \leq w_{n}U
\Rightarrow |\tilde{\psi}_{n}(U)|\leq w_{n}\frac{U^{2}}{2}.
\end{split}
\end{equation}
Now, concerning the $\tilde{\phi}_{n}(U)$ we exploit the same arguments as before and by (\ref{inth}) we obtain
\begin{equation}\label{fifi}
\tilde{\phi}'_{n}(U)=\int \frac{x\sin(Ux)}{x^{2}}|H_{n}(x)|dx\leq \int \frac{|U|x}{x}|H_{n}(x)|dx\leq C|U|\left(1 + w_{n}^{\frac{r-1}{r-2}}\right)
\end{equation}
\begin{equation}\label{fi}
\tilde{\phi}_{n}(U)=\int \frac{1-\cos(Ux)}{x^{2}}|H_{n}(x)|dx\leq \int \frac{(Ux)^{2}}{x^{2}}|H_{n}(x)|dx
\leq CU^{2}\left( 1 + w_{n}^{\frac{r-1}{r-2}} \right).
\end{equation}
%%%%%%%%%%%%%%%%%%%%%%%%%%%%%%%%%%%%%%%%%%%%%%%%%%%%%%%%
\subsection{Total variation distance}
In order to establish a lower bound for our class with the rates (\ref{rates}), the last ingredient to be shown is that the total variation distance between $\Pm_{\textbf{X}}$ and $\Pm_{\textbf{Y}}$ goes towards zero, property \ref{prop2}. Mathematically speaking, this formulates as
\begin{equation}\label{tota}
\tv(\Pm_{\textbf{X}}, \Pm_{\textbf{Y}})=2n\int \left(f_{1/n}(x)-g_{1/n}(x)\right)dx\to 0.
\end{equation} 
As we discussed in the first step, $\textbf{X}$ and $\textbf{Y}$ have a nonvanishing Gaussian part so that the variables $\textbf{X}_{1/n}$ and $\textbf{Y}_{1/n}$ have densities. Here, $f_{1/n}$ and $g_{1/n}$ denote their densities respectively. Also, $k_{n}=f_{1/n}-g_{1/n}$ denotes the difference between the densities. One would be tempted to use the following 
\begin{equation}\label{tv}
\begin{aligned}
\tv(\Pm_{0}, \Pm_{1})&=2 n\int \Big(f_{1/n}(x)-g_{1/n}(x)\Big)dx\\
&=2 n\int \left( \int e^{-iUx}\big(\phi_{n}(U)-\psi_{n}(U)\big) dU\right)dx\\
&\leq 2 n \int \left( \int |e^{-i Ux}| |\phi_{n}(U)-\psi_{n}(U)|dU\right)dx\\
&\leq 2 n\int \left( \int  |\phi_{n}(U)-\psi_{n}(U)|dU\right)dx.
\end{aligned}
\end{equation}
In the second equality we wrote the density function as the Inverse Fourier transform of its characteristic function. But the last integral is infinite. Hence, this procedure is not working for our goal. Since we want to prove that 
\begin{equation}
2 n \int  |\phi_{n}(U)-\psi_{n}(U)|dU\to 0,
\end{equation}
we know that the total variation distance between $\Pm_{\textbf{X}}$ and $\Pm_{\textbf{Y}}$ is not more than $2n$ times $\int |k_{n}(x)|dx$. 
By using the same argument as for the \cite{jacod2014remark} Theorem 3.1, by Cauchy-Schwarz inequality and Plancherel theorem, we obtain  
\[
\begin{aligned}
\int|k_{n}(x)|dx&=\int\frac{1}{\sqrt{1+|x|^2}}(\sqrt{1+|x|^2})|k_{n}(x)|dx\\
&\leq K\bigg(\int \left(k_{n}^{2}(x)+|x|^{2}k_{n}^{2}(x)\right)dx\bigg)^{1/2}\\
&\leq K\left(\|\eta_{n}\|_{\mathbb{L}^{2}}+ \|\partial _{1}\eta_{n}\|_{\mathbb{L}^{2}}\right)^{1/2}
\end{aligned}
\]
where $\partial_{1}\eta_{n}$ is the first derivative of $\eta_{n}(U)$. In the second inequality we used the Cauchy-Schwarz inequality, and in the last one we used Plancherel identity. By virtue of simplicity, remember that we use the same coordinates for the vector $\textbf{u}=(U,U)$.\par 
Thus, the only ingredient which remains to be shown is the following lemma in order to satisfy Property \ref{prop2}.

\begin{lemma}\label{totava}
	We show that 
	\begin{equation}\label{totalvar}
	4n^{2}\int |\eta_{n}(U)|^{2}dU \to 0 \quad \mbox{and}\quad 4n^{2}\int |\partial_{1}\eta_{n}(U)|^{2}dU \to 0
	\end{equation}
	as $n\to \infty$. 
\end{lemma}
\begin{proof}
	First, we study the convergence of $\eta_{n}(U) $:
	\begin{equation}\label{eta}
	\begin{split}
	|\eta_{n}(U)|=&|\phi_{n}(U)-\psi_{n}(U)|=\cfx \bigg|1-\frac{\cfy}{\cfx} \bigg|\\
	&=\cfx\bigg| 1-\exp\left( \frac{1}{2n}(2\cfxi-w_{n}U^{2})\right) \bigg|\\
	&\leq \cfx\bigg|\frac{1}{2n}\left(2\cfxi-w_{n}U^2\right)\bigg|.
	\end{split}
	\end{equation}
	The last inequality holds due to the fact that $1-e^{-x}\leq x $. \par 
	Observe that when $|U|\leq U_{n} $, $ \eta_{n}(U)=0 $ because of the constant value of $h_{n}$ inside this interval. Thus the difference of the characteristic functions vanishes for $ |U|\leq U_{n} $ because of $ 2\cfxi=w_{n}U^2 $ by (\ref{psiti}). \par 
	By means of $\tilde{\phi}_{n}(U), \tilde{\psi}_{n}(U)\geq 0 $, we get that $ \cfx \leq e^{-\frac{U^{2}}{2n}}$ and  $ \cfy \leq e^{-\frac{U^{2}}{2n}}$. Thus, 
	\[
	|\eta_{n}(U)|\leq \frac{U^{2}w_{n}}{2n}e^{-\frac{U^{2}}{2n}}\mathds{1}_{\left\{|U|\geq u_{n}\right\}}.
	\]
	We define by
	\begin{equation}\label{int}
	A := \int_{\left\{|U|\geq U_{n}\right\}}U^{4}e^{-\frac{U^{2}}{n}} dU.
	\end{equation}
	Using Cauchy-Schwarz inequality we bound the integral (\ref{int}) from above by
	\begin{equation}\label{a}
	A \leq 2 \Bigg(\int_{U_n}^{\infty} U^{5}e^{-\frac{U^{2}}{n}} dU \Bigg)^{1/2}\cdot\Bigg(\int_{U_n}^{\infty} U^{3}e^{-\frac{U^{2}}{n}} dU \Bigg)^{1/2}.
	\end{equation}
	The integrals to the right can be calculated exactly by calculus methods, and recalling $ U_{n}= 2w_{n}^{1/(r-2)} $ we get 
	\begin{equation}
	4n^{2}\int |\eta_{n}(U)|^{2}dU \leq 4n^{2} \int_{\left\{|U|\geq U_{n}\right\}}U^{4}e^{-\frac{U^{2}}{n}} dU\leq C\frac{(\log n)^{3/2}}{n^{3/2}}.
	\end{equation}
	The first part of the (\ref{totalvar}) follows. Now recall the form of the characteristic functions and their difference $ \eta_{n} = \psi_{n}(U) - \phi_{n}(U) $
	\[
	\psi_{n}(U) = \exp\left\{-\frac{1}{2n}\left(5U^{2} + 2 \tilde{\phi}_{n}(U) +2 \tilde{\phi}_{n}(U)\right)\right\}
	\]
	\[
	\phi_{n}(U) = \exp\left\{-\frac{1}{2n}\left(5U^{2}  +4 w_{n}U^{2} + 2 \tilde{\phi}_{n}(U) \right)\right\}.
	\]
	Therefore by (\ref{psitil}), (\ref{fifi}), and the fact that $ \cfx \leq e^{-\frac{U^{2}}{2n}}$,  $ \cfy \leq e^{-\frac{U^{2}}{2n}}$ we get that 
	\begin{equation}
	\begin{split}
	|\partial_{1}\eta_{n}(U)|&=\frac{1}{n}\bigg|\left(5U +  4w_{n}U + \tilde{\phi}_{n}'(U)\right)\phi_{n}(U) - \left(5U + \tilde{\phi}_{n}(U) + \tilde{\psi}_{n}'(U)\right)\psi_{n}(U)\bigg|\\
	&=\frac{1}{n}\bigg| 4w_{n}U\phi_{n}(U)+\left(\tilde{\phi}_{n}'(U)+5U\right)\eta_{n}(U)-\tilde{\psi}_{n}'(U)\psi_{n}(U)\bigg|\\
	&\leq \frac{1}{n}\left(4 w_{n}|U|e^{-\frac{U^{2}}{2n}} +|\tilde{\phi}_{n}'(U)+5U|e^{-\frac{U^{2}}{2n}}\frac{w_{n}U^{2}}{2n} - w_{n}|U|e^{-\frac{U^{2}}{2n}} \right)\\
	&\leq C\frac{w_{n}|U|}{n}e^{-\frac{U^{2}}{2n}}\left(1+\left(w_{n}^{\frac{r-1}{r -2}}+1\right)\frac{U^{2}}{2n}\right).
	\end{split}
	\end{equation}
	Therefore,
	\begin{equation}\label{con}
	\begin{aligned}
	4n^{2}\int_{U_{n}}^{\infty}	|\partial_{1}\eta_{n}(U)|^{2} \leq & Cw_{n}^{2} \int_{U_{n}}^{\infty}U^{2}e^{-\frac{U^{2}}{n}}dU + C\frac{w_{n}^{2}}{n^{4}}\left(1 + w_{n}^{\frac{r-1}{r-2}}\right)^{2} \int_{U_n}^{\infty} U^{6}e^{-\frac{U^{2}}{n}}dU \\ 
	&  + C\frac{w_{n}^{2}}{n}\left(1 + w_{n}^{\frac{r-1}{r-2}}\right)^{2}\int_{U_n}^{\infty} U^{4}e^{-\frac{U^{2}}{n}}dU.
	\end{aligned}
	\end{equation}
	Now, $\frac{1 + w_{n}^{\frac{r-1}{r-2}}}{n^{2}}$ and $\frac{1 + w_{n}^{\frac{r-1}{r-2}}}{n}$  tend towards zero. Additionally, the integrals on the right side can be bounded again using Cauchy-Schwarz inequality, like integral A. Following these, we can calculate the integrals through basic calculus methods exactly. As a result,
	\begin{equation}\label{part}
	4n^{2}\int_{U_{n}}^{\infty} |\partial_{1}\eta_{n}(U)|^{2}dU \leq C \frac{(\log n)^{2}}{n^{1/2}},
	\end{equation}
	which also goes to zero as $ n\to \infty $ and the proof is completed. 
\end{proof}
\subsection* {End proof of \autoref{lobo}}

\textit{Lower bound for  the rate $w_{n}= \frac{1}{\sqrt{n}}$ when $r\in(0,2)$.}\\
To prove this bound, it is enough to show that it holds on the subclass of all Brownian motions since $\mathcal{B}_{M}\supset \class$. Taken together with Lemma \ref{twodim}, this bound is achieved.\\
\textit{Lower bound for the rate $w_{n}=1/(n\log n)^{\frac{2-r}{2}}$ when $r\in (1,2)$}.\\
The main steps of this proof are to show that Property \ref{prop1} and Property \ref{prop2} are satisfied. Now, with reference to Lemma \ref{diagon} for the construction of co-jump measure, Proposition \ref{prob} and Lemma \ref{totava}  we conclude the  proof of \autoref{lobo}.

%%%%%%%%%%%%%%%%%%%%%%%%%%%%%%%%%%%%%%%%%%%%%%%%%%%%%%%%%%%%%%%
\section{Discussion}\label{sec:remarks}
In this section we make some important remarks concerning the upper bound and the rates of convergence.
First, we want to compare the efficiency of our estimator with the work of \cite{mancini2017truncated} in which she considered at least one jump component of a two-dimensional It\^o semimartingale with infinity variation. \cite{mancini2017truncated} introduced the truncated realized covariance (TRC) as an estimator for co-integrated volatility. The proposed estimator is 
\[
\widehat{IC}=\sum_{i=1}^{n}\Delta_{i} X^{(1)}\ch_{\{(\Delta_{i} X^{(1)})^2\leq r_h\}} \Delta_{i} X^{(2)}\ch_{\{(\Delta_{i} X^{(2)})^2\leq r_h\}},
\]
where $ r_{h} = h^{2u}$ is the truncation level with $ h = 1/n $, $ u \in(0,\frac{1}{2}) $ and $ n\to \infty $. It is clear that, when $ r_{h} \to 0 $, asymptotically all jumps are excluded. It is assumed that the two jump components have an index activity $r_{1}$,$r_{2}$ where $ 0\leq r_{1}\leq r_{2}<2$ and $ r_{2}\geq 1 $. Notice by recalling \Cref{coj} that in our case we used the index r for the activity of co-jumps. In the following we use the notation ``$ \gg $''  to assume the index activity is much greater than $ 1 $ and close to $ 2 $. The truncated estimator achieves the rate $(1-\gamma)\sqrt{r_{h}}^{(1+\frac{r_{2}}{r_{1}}-r_{2})}$ when $ r_{1} ,r_{2}\gg 1 $. This estimator reaches the rate $h\sqrt{r_{h}}^{-\frac{r_{2}}{2}}$ when the two jump components are independent or $ r_{2}\gg r_{1} $ and $ r_{1} $ is small and the rate $ \sqrt{h} $ when $ r_{1} $ is small and $ r_{2} $ is close to $ 1 $. The parameter $\gamma$ describes the dependence structure of the two jumps with $\gamma\in[0,1]$. When $\gamma=0$ we have full dependency between the jump components, while $\gamma=1$ means independence between the jump components. Finally, for a fair comparison with the spectral estimator we assume $\sqrt{r_{h}}$ to be approximately $\frac{1}{\sqrt{n}}$ as truncation level, since $u\in(0,\frac{1}{2})$.  The truncation level is not optimal, but the work of \cite{figueroa2017optimum} proposed an optimal way for the truncation level using mean and conditional mean square error for the case of a one-dimensional It\^o semimartingale. \par
The reliability of estimators (\textit{spectral, truncated}) is summarized and assessed in the following table. For simplicity we take into consideration only the two extreme cases of dependency. In the first two rows we assume $ \gamma=0 $, i.e., the dependency setting among the jump components. While in the last row we consider $ \gamma=1 $, i.e., the jumps are totally independent. In order to compare the estimators we use \Cref{coj} and \Cref{example}.\\

\begin{tabular}{ |p{4cm}||p{3cm}|p{3cm}|  }
	\hline
	\multicolumn{3}{|c|}{Rates of convergence} \\
	\hline
	$\textcolor{darkgreen}{r, r_1, r_2 \in[0,2)} $& \textcolor{blue}{TRC estimator} &\textcolor{blue}{Spectral estimator}\\
	\hline
	$r_{2}$ close to $ 1 $, $ r_{1} $ small &$ n^{-\frac{1}{2}}$&  $n^{-\frac{1}{2}}$  \\
	$r, r_{1},r_{2} $ close to $ 2 $& $n^{-\frac{1}{2}\big(1+\frac{r_{2}}{r_{1}}-r_{2}\big)}$&  $(n\log n)^{\frac{r-2}{2}}$ \\
	$r_{1}$ small, $ r_{2} $ close to $ 2 $ &$ n^{\frac{r_{2} - 2}{2}}$& $ (n\log n)^{\frac{r-2}{2}}$ \\
	Independent jumps &  $ n^{\frac{r_{2}-2}{2}}$& $(n\log n)^{\frac{r-2}{2}} $\\
	\hline
\end{tabular}

\vspace{0.5cm}
Here we notice that when we assume a dependence structure among the jump components the spectral estimator achieves faster rates than the truncated estimator, when we assume infinite variation for both jump components. Notwithstanding, the truncated estimator establishes same rates with the spectral estimator when both jump components have index activity close to $ 1 $. Finally, when we assume either independence between jump components or $ r_1 $ is much smaller than $ r_{2} $ then the spectral estimator reaches a faster rate.
%%%%%%%%%%%%%%%%%%%%%%%%%%%%%%%%%%%%%%%%%%%%%%%%%%%%%%%%%%%%%%%%%
\subsection{Numerical experiments.}
In this section we test our estimates with Monte-Carlo experiments.\footnote{The interested reader can view the code at \url{https://github.com/KarinaPapayia/Co-integrated-volatility-multidimensional-Levy-processes}} This means that we first have to simulate the sample paths of a bivariate L\'evy process on $ [0, 1] $. \par
Section 6 of \cite{tankov2003financial} suggested various simulation algorithms for L\'evy processes. We extend here Algorithms 6.6, 6.5, 6.3 to a bivariate setting. In addition, we use the generalized shot noise method for series representation of a two-dimensional L\'evy process of infinite variation introduced by \cite{rosinski1990}. \par

We now perform Monte-Carlo tests of our spectral estimate $\widehat{C}_{12}^{N}(U_N) $, comparing it to the Truncated Realized Covariance (TRC) estimate $ \widehat{IC}_{T} $ of \cite{mancini2017truncated} for a two-dimensional It\^o semimartingale. To provide a balanced comparison, we will draw our observations from a process $ X_t = B_t + J_t$, where $ B_t $ is a two-dimensional Brownian motion and $ J_t $ is a two-dimensional jump process. Its jumps are driven by a two-dimensional $ r $-stable process. $ X_t $ thus models a process with both diffusion and jump components. In each run of our simulation, we will generate $ N = 1,000 $ observations, corresponding to observations taken every $ 1/1,000 $ over a time interval $ [0, 1] $. \par
\begin{figure}[H]
	\includegraphics[width=9cm]{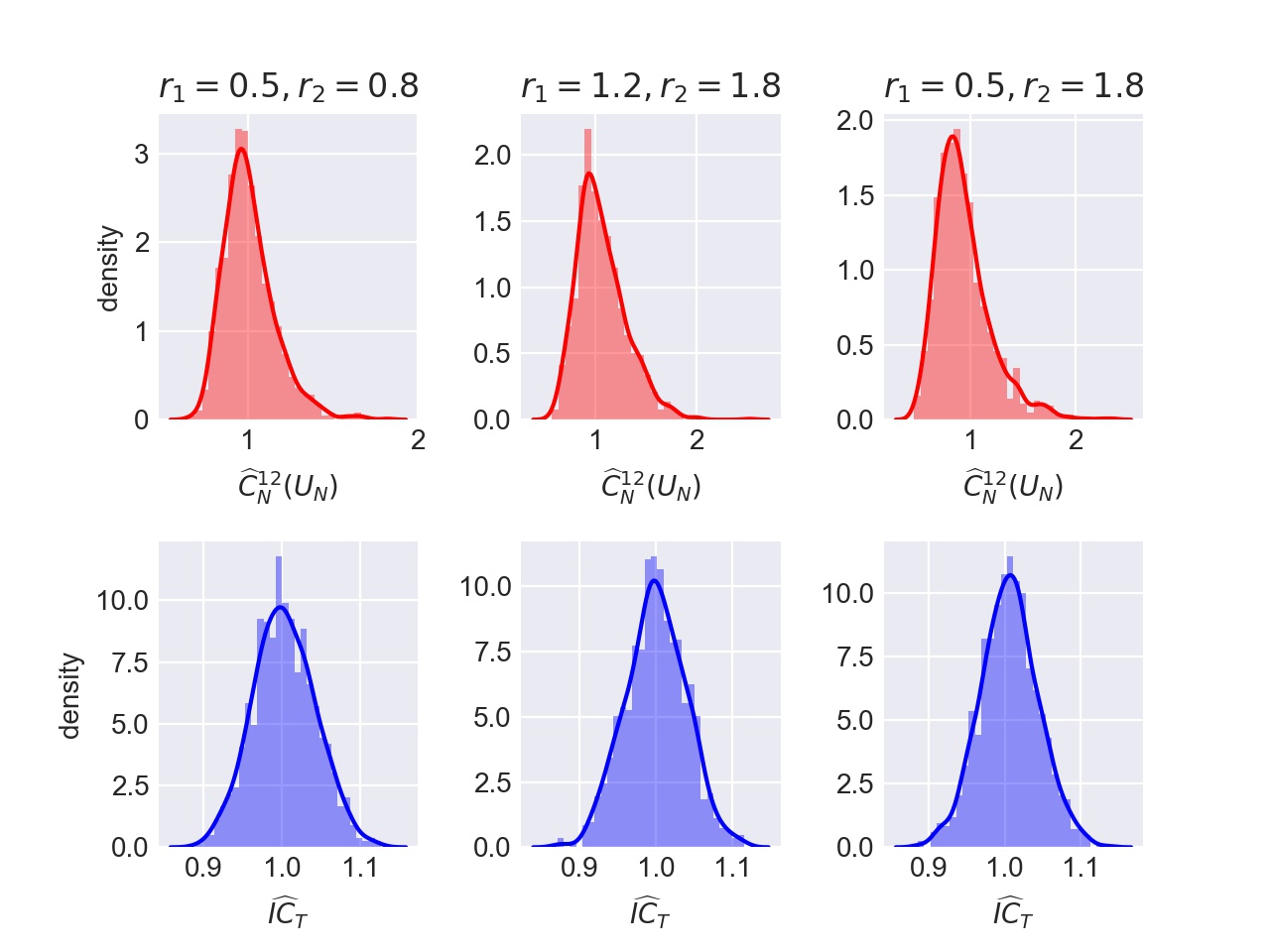}
	\caption{Simulated distributions of the estimates.}
	\label{fig:hist}
\end{figure}
The estimates $ \widehat{C}^{12}_{N}(U_N) $ and $ \widehat{IC}_{T} $ depend on a number of parameters. We begin by considering the covariance matrix $C = \left( \begin{smallmatrix} 2 & 1\\1 & 1\end{smallmatrix}\right) $ for two correlated Brownian motions. In our simulations, the cointegrated volatility of $ X_t $ is equal to $ 1 $, and so we may choose the parameters accordingly. In our tests, we found the value $ M = 4.229 $ worked well for bounding from above the jump activity in the case of infinite variation jumps. In the case of $\widehat{IC}_{T} $ we chose $ h = 1/1,000 $, $ u = 0.387 $, and as truncation level $ r_{h} = \left(\frac{1}{1,000}\right)^{2 * 0.387} $. We found that this truncation level cuts the jumps bigger than $\left(\frac{1}{1,000}\right)^{2 * 0.387} $, which means that almost all jumps were eliminated.\par
Figure \ref{fig:hist} plots the simulated distributions of the estimates $ \widehat{C}_{N}^{12}(U_{N}) $ and $ \widehat{IC}_{T}$ together with the density of a standard Gaussian distribution, shown as a solid line. We can see that in every choice for $ r_{1}, r_{2} $, the estimates are centered around $ 1 $, which is the expected theoretical cointegrated volatility.\par
Figure \ref{fig:rmse1} plots the RMSEs of the estimates $\widehat{C}_{N}^{12}(U_{N}) $, $\widehat{IC}_{T} $ against different choices for the index activity of the co-jumps. We study the performance of the estimates under finite, moderate, and infinite activity of co-jumps. We can see that, as $ N $ grows, the RMSE of the spectral estimate $\widehat{C}_{N} ^{12}(U_{N})$ is getting smaller compared with the truncated estimate. However, we observe that the RMSEs of the truncated estimate $ \widehat{IC}_{T} $ are smaller compared with the spectral estimate when $ N=1,000 $. \par
\begin{figure}[H]
	\includegraphics[width=9cm]{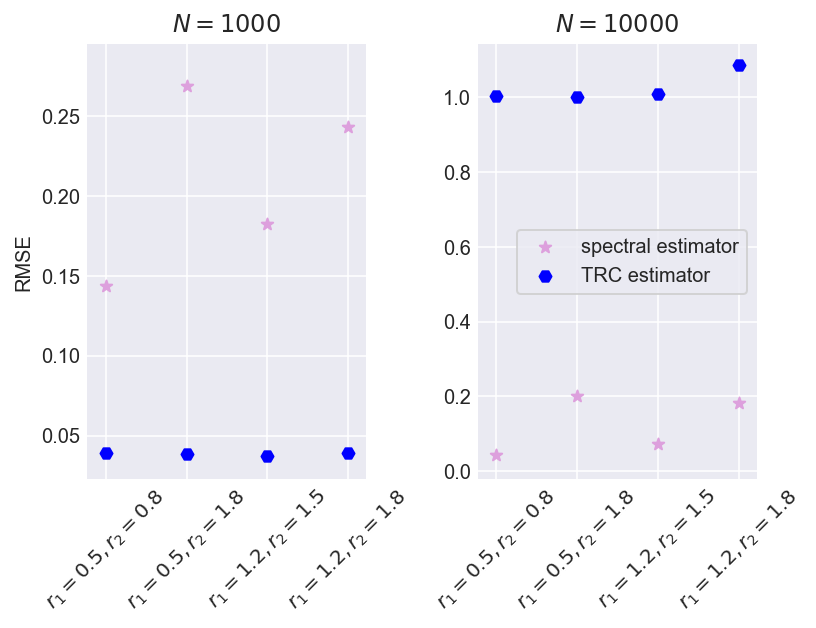}
	\caption{Simulated RMSEs of the estimates $\widehat{C}_{N}^{12}(U_{N}) $ and $  \widehat{IC}_{T}$. }
	\label{fig:rmse1}
\end{figure}
We observe this behavior in Figure \ref{fig:rmse1} for the truncated estimate $\widehat{IC}_{T} $ because of our choice of truncation level, which is not an optimal. While the threshold $r_{h} = \left(\frac{1}{1,000}\right)^{2 * 0.387}  $ works well for $ N = 1,000 $, it does not work well when the number of observations is bigger, for example when $ N = 10,000 $.\par

Figures \ref{fig:vp1}, \ref{fig:vp2}  give violin plots for the spectral estimate $ \widehat{C}_{N}^{12}(U_{N}) $ under a number of choices for the amount of observations $ N $ and the index activity for the co-jumps, whilst Figures \ref{fig:vp3}, \ref{fig:vp4} show violin plots for the truncated estimate $ \widehat{IC}_{T} $ under the same settings. The number of observations varies from $ 1,000 $ to $ 10,000 $ by step $1,000$. \par
\begin{figure}[H]
	\includegraphics[width=10cm]{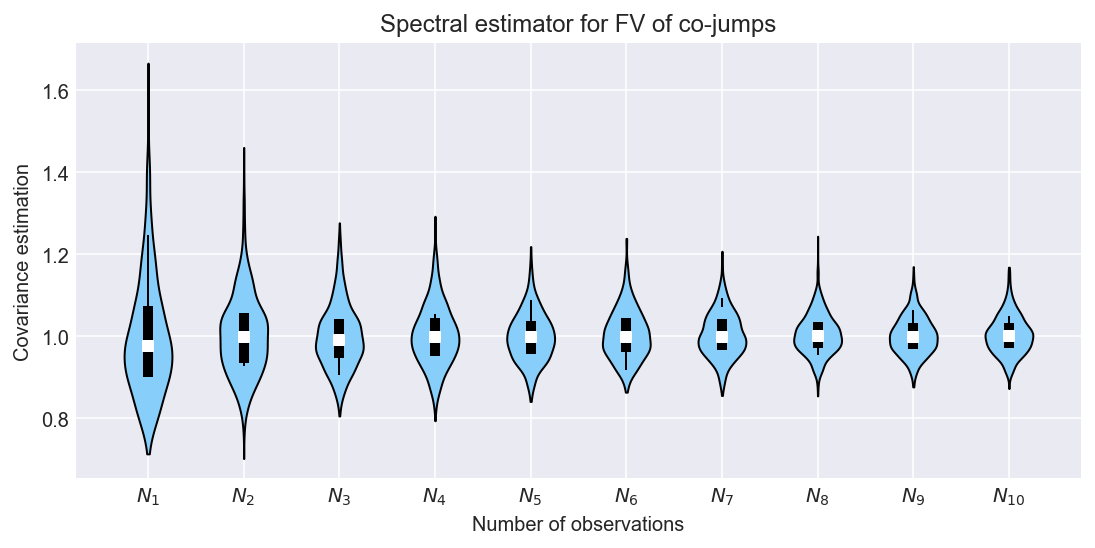}
	\caption{Violin plots for the estimates $\widehat{C}_{N}^{12}(U_{N}) $.} 
	\label{fig:vp1}
\end{figure}
\begin{figure}[H]
	\includegraphics[width=10cm]{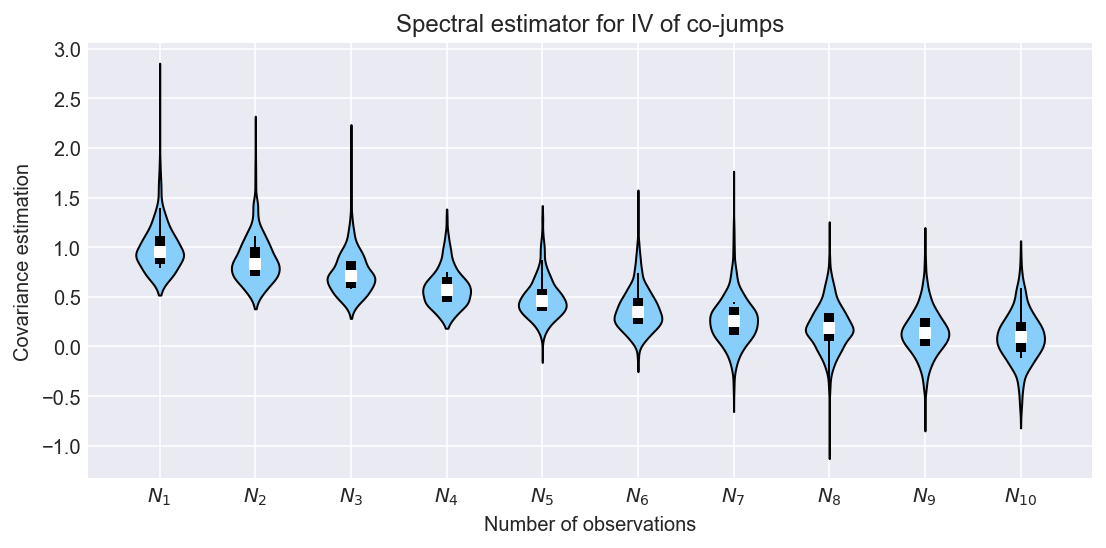}
	\caption{Violin plots for the estimates $\widehat{C}_{N}^{12}(U_{N}) $.} 
	\label{fig:vp2}
\end{figure}
In Figure  \ref{fig:vp1}, we used as an index activity for the jumps $ r_{1} =0.5 $, $ r_{2} = 0.8 $, while in Figure \ref{fig:vp2} we set $ r_{1} = 1.2 $ and $ r_{2} = 1.8$. In the case of  $ r_{1} = 1.2 $ and $ r_{2} = 1.8$, we see that the estimation for the covariance 
slightly deviates from the center as $ N $ grows. Furthermore, the effect can be expected to disappear as $ N $ tends toward infinity.\par
\begin{figure}[H]
	\includegraphics[width=10cm]{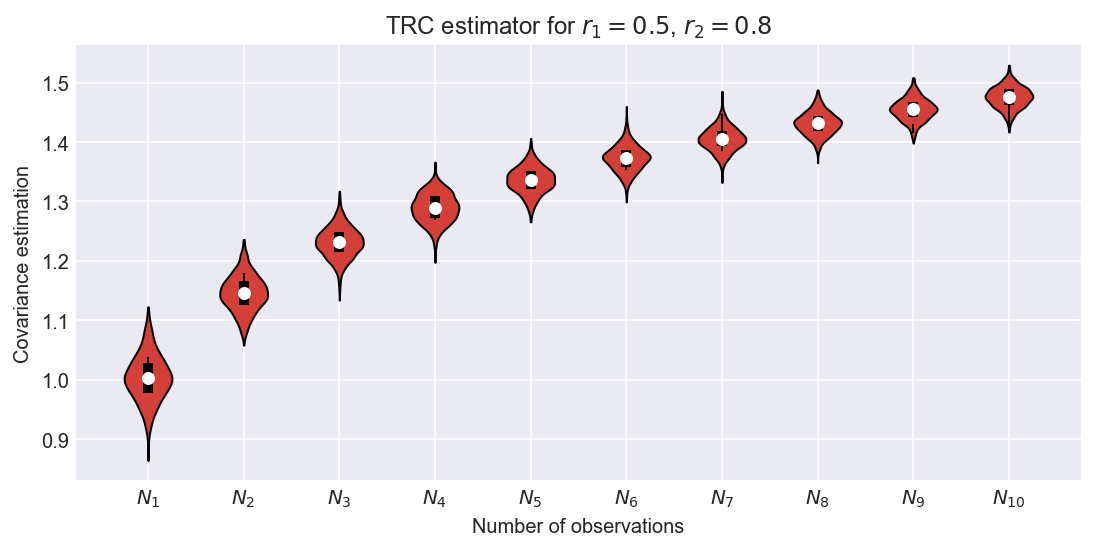}
	\caption{Violin plots for the estimates $\widehat{IC}_{T}$.} 
	\label{fig:vp3}
\end{figure}
\begin{figure}[H]
	\includegraphics[width=10cm]{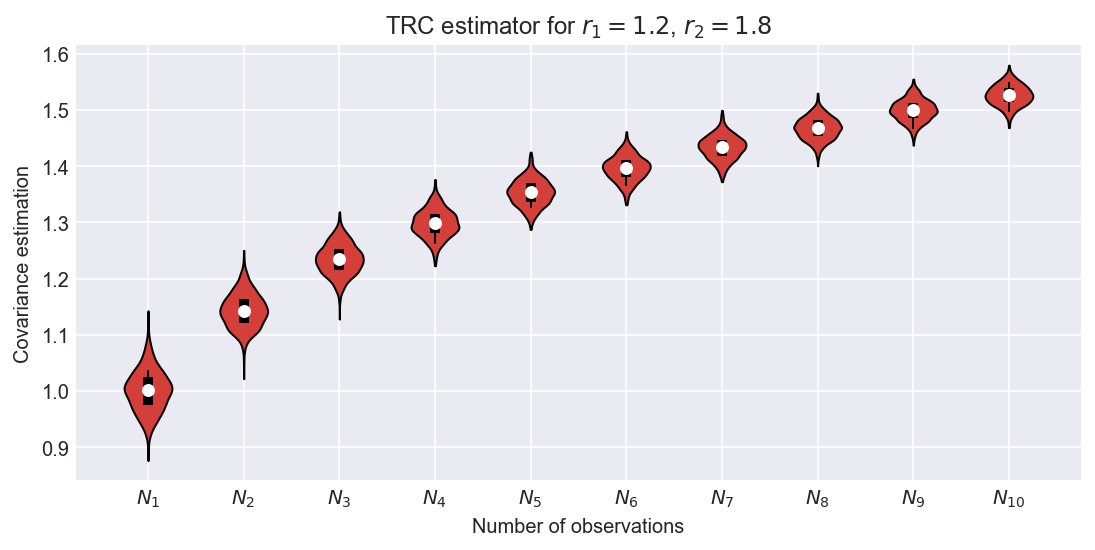}
	\caption{Violin plots for the estimates $\widehat{IC}_{T}$.} 
	\label{fig:vp4}
\end{figure}
Figures \ref{fig:vp3} and \ref{fig:vp4} show again that the truncated estimate $ \widehat{IC}_{T} $ deviates strongly from the center as $ N $ grows, an expected effect due to the choice of truncation level. The chosen threshold works well for $ N=1,000 $ but not when $ N $ grows. We expect this effect to disappear once the optimal choice for the threshold $ r_{h} $ is established. Finally, $ U_{N} $ is the parameter which controls the frequency for our spectral estimate $ \widehat{C}_{N}^{12}(U_{N}) $. $ U_{N} $ depends on $ N, M, r $. In view of the form (\ref{you}) for $ U_{N} $ we can find a constant to multiply which will give us the optimal choice for $ U_{N} $. The results will still hold. In fact \ref{fig:M} shows that the spectral estimate for $ M >3.31 $, $ N = 5,000 $ and $ r = 1.5 $ is centered around the theoretical co-integrated volatility $ C^{12} $. Figure \ref{fig:M} shows violin plot for the spectral estimate tuning up the parameter $ M $, which ranges from $ 3.30 $ to $ 4.40 $ .  \par
\begin{figure}[H]
	\includegraphics[width=10cm]{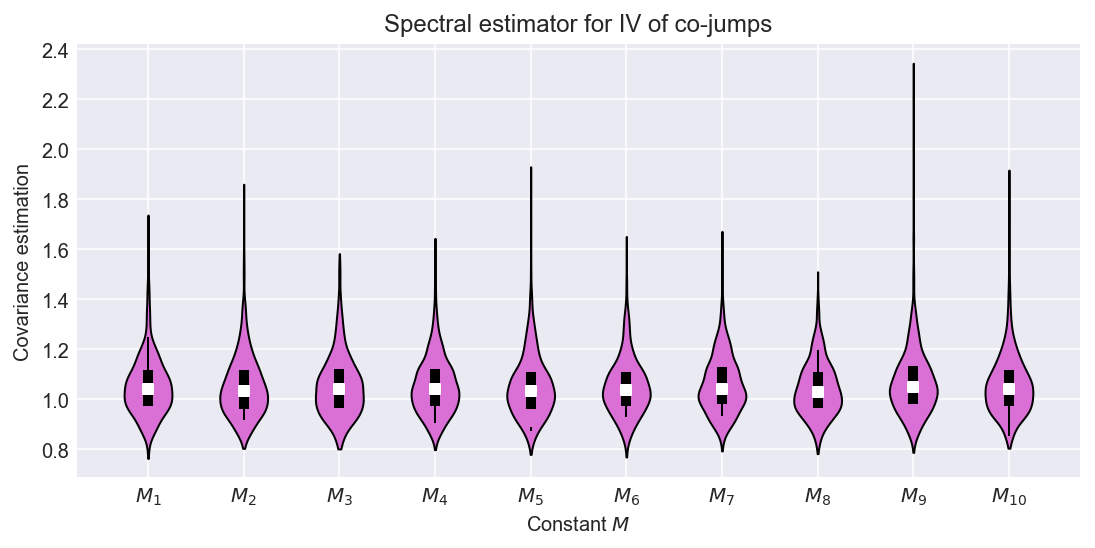}
	\caption{Tuning parameter $M  $ for the estimates $\widehat{C}_{N}^{12}$.} 
	\label{fig:M}
\end{figure}
Figure \ref{fig:U} shows that the truncated estimate $ \widehat{IC}_{T} $ is quite sensitive to the choice of threshold. Here, we used $N= 5,000$, $ r_{1}= 1.2$, $ r_{2} = 1.5$, $ h = 1/5,000 $ and $ u $ varies from $ 0.41 $ to $ 0.42 $. Recall that $ r_{h}= h^{2u} $.
\begin{figure}[H]
	\includegraphics[width=10cm]{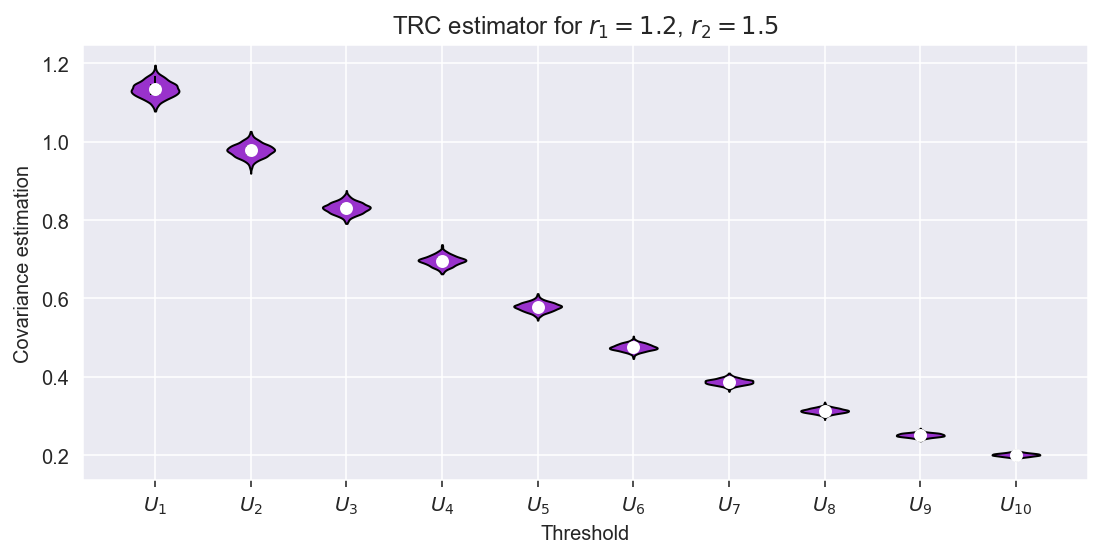}
	\caption{Tuning parameter $ r_{h} $ for the estimates $\widehat{IC}_{T}$.} 
	\label{fig:U}
\end{figure}
We notice that the threshold estimate deviates strongly from the theoretical co-integrated volatility. Figure \ref{fig:U} shows that the threshold estimate is centered around $ C^{12} $ when $ u = 0,412 $. The optimal choice of the threshold is not trivial. To sum up, it is easier to tune up the parameters for the spectral estimate rather than the threshold for the truncated estimate.
%%%%%%%%%%%%%%%%%%%%%%%%%%%%%%%%%%%%%%%%%%%%%%%%%%%%%%%%%%
\section*{Acknowledgement}
The author is very grateful to an anonymous referee for many helpful questions and remarks that have led to considerable improvements and to Markus Rei$\ss$ for stimulating comments and discussions. This work was partially supported by Deutsche Forschungsgemeinschaft via IRTG 1792 \textit{High Dimensional Nonstationary Time Series}.

\bibliography{bibcov}{}

\begin{thebibliography}{40}
\providecommand{\natexlab}[1]{#1}
\providecommand{\url}[1]{\texttt{#1}}
\expandafter\ifx\csname urlstyle\endcsname\relax
  \providecommand{\doi}[1]{doi: #1}\else
  \providecommand{\doi}{doi: \begingroup \urlstyle{rm}\Url}\fi

\bibitem[A{\"\i}t-Sahalia and Jacod(2011)]{ait2011testing}
Y.~A{\"\i}t-Sahalia and J.~Jacod.
\newblock Testing whether jumps have finite or infinite activity.
\newblock \emph{The Annals of Statistics}, 2011.

\bibitem[Andersen and Bollerslev(1998)]{andersen1998answering}
T.~Andersen and T.~Bollerslev.
\newblock Answering the skeptics: Yes, standard volatility models do provide
  accurate forecasts.
\newblock \emph{International economic review}, 1998.

\bibitem[Barndorff-Nielsen and Shephard(2002)]{barndorff2002}
O.~Barndorff-Nielsen and N.~Shephard.
\newblock Econometric analysis of realized volatility and its use in estimating
  stochastic volatility models.
\newblock \emph{Journal of the Royal Statistical Society: Series B (Statistical
  Methodology)}, 2002.

\bibitem[Barndorff-Nielsen and Shephard(2004)]{barndorff2004}
O.~Barndorff-Nielsen and N.~Shephard.
\newblock Power and bipower variation with stochastic volatility and jumps.
\newblock \emph{Journal of financial econometrics}, 2004.

\bibitem[Barndorff-Nielsen and Shephard(2006)]{barndorff2006}
O.~Barndorff-Nielsen and N.~Shephard.
\newblock Econometrics of testing for jumps in financial economics using
  bipower variation.
\newblock \emph{Journal of financial Econometrics}, 2006.

\bibitem[Belomestny(2010)]{belomestny2010spectral}
D.~Belomestny.
\newblock {Spectral estimation of the fractional order of a L{\'e}vy process}.
\newblock \emph{The Annals of Statistics}, 2010.

\bibitem[Belomestny and Panov(2013{\natexlab{a}})]{belomestny2012abelian}
D.~Belomestny and V.~Panov.
\newblock Abelian theorems for stochastic volatility models with application to
  the estimation of jump activity of volatility.
\newblock \emph{Stochastic Processes and their Applications},
  2013{\natexlab{a}}.

\bibitem[Belomestny and Panov(2013{\natexlab{b}})]{belomestny2013estimation}
D.~Belomestny and V.~Panov.
\newblock {Estimation of the activity of jumps in time-changed L{\'e}vy
  models}.
\newblock \emph{Electronic Journal of Statistics}, 2013{\natexlab{b}}.

\bibitem[Belomestny and Rei{\ss}(2006)]{belomestny2006spectral}
D.~Belomestny and M.~Rei{\ss}.
\newblock {Spectral calibration of exponential L{\'e}vy models}.
\newblock \emph{Finance and Stochastics}, 2006.

\bibitem[Belomestny and Rei{\ss}(2015)]{belomestny2015estimation}
D.~Belomestny and M.~Rei{\ss}.
\newblock Estimation and calibration of l{\'e}vy models via fourier methods.
\newblock In \emph{L{\'e}vy Matters IV}. Springer, 2015.

\bibitem[Belomestny and Trabs(2018)]{MR3825892}
D.~Belomestny and M.~Trabs.
\newblock {Low-rank diffusion matrix estimation for high-dimensional
  time-changed {L}\'{e}vy processes}.
\newblock \emph{Ann. Inst. Henri Poincar\'{e} Probab. Stat.}, 2018.

\bibitem[Bibinger and Vetter(2015)]{bibinger2015estimating}
M.~Bibinger and M.~Vetter.
\newblock Estimating the quadratic covariation of an asynchronously observed
  semimartingale with jumps.
\newblock \emph{Annals of the Institute of Statistical Mathematics}, 2015.

\bibitem[Bibinger and Winkelmann(2015)]{bibinger2015econometrics}
M.~Bibinger and L.~Winkelmann.
\newblock Econometrics of co-jumps in high-frequency data with noise.
\newblock \emph{Journal of Econometrics}, 2015.

\bibitem[Bibinger et~al.(2014)Bibinger, Hautsch, Malec, and
  Rei{\ss}]{bibinger2014estimating}
M.~Bibinger, N.~Hautsch, P.~Malec, and M.~Rei{\ss}.
\newblock Estimating the quadratic covariation matrix from noisy observations:
  Local method of moments and efficiency.
\newblock \emph{The Annals of Statistics}, 2014.

\bibitem[Bracewell(1986)]{bracewell1986fourier}
R.~Bracewell.
\newblock \emph{The Fourier transform and its applications}, volume 31999.
\newblock McGraw-Hill New York, 1986.

\bibitem[B{\"u}cher and Vetter(2013)]{bucher2013nonparametric}
A.~B{\"u}cher and M.~Vetter.
\newblock {Nonparametric inference on L{\'e}vy measures and copulas}.
\newblock \emph{The Annals of Statistics}, 2013.

\bibitem[Carr et~al.(2002)Carr, Geman, Madan, and Yor]{carr2002}
P.~Carr, H.~Geman, D.~Madan, and M.~Yor.
\newblock The fine structure of asset returns: An empirical investigation.
\newblock \emph{The Journal of Business}, 2002.

\bibitem[Christensen et~al.(2013)Christensen, Podolskij, and Vetter]{pod1}
K.~Christensen, M.~Podolskij, and M.~Vetter.
\newblock {On covariation estimation for multivariate continuous It{\^o}
  semimartingales with noise in non-synchronous observation schemes}.
\newblock \emph{Journal of Multivariate Analysis}, 2013.

\bibitem[Coca(2018)]{coca2018efficient}
A.~Coca.
\newblock Efficient nonparametric inference for discretely observed compound
  poisson processes.
\newblock \emph{Probability Theory and Related Fields}, 2018.

\bibitem[Comte and Genon-Catalot(2014)]{comte2014nonparametric}
Duval~C. Comte, F. and V.~Genon-Catalot.
\newblock {Nonparametric density estimation in compound Poisson processes using
  convolution power estimators}.
\newblock \emph{Metrika}, 2014.

\bibitem[Comte and Genon-Catalot(2009)]{comte2009nonparametric}
F.~Comte and V.~Genon-Catalot.
\newblock {Nonparametric estimation for pure jump L{\'e}vy processes based on
  high frequency data}.
\newblock \emph{Stochastic Processes and their Applications}, 2009.

\bibitem[Duval and Mariucci(2017)]{duval2017compound}
C.~Duval and E.~Mariucci.
\newblock {Compound Poisson approximation to estimate the L{\'e}vy density}.
\newblock \emph{arXiv preprint arXiv:1702.08787}, 2017.

\bibitem[Eberlein and Papapantoleon(2005)]{eberlein2005}
E.~Eberlein and A.~Papapantoleon.
\newblock {Symmetries and pricing of exotic options in L{\'e}vy models}.
\newblock \emph{Exotic option pricing and advanced L{\'e}vy models}, 2005.

\bibitem[Figueroa-Lopez and Houdr{\'e}(2004)]{figueroa2004nonparametric}
E.~Figueroa-Lopez and C.~Houdr{\'e}.
\newblock {Nonparametric estimation for L{\'e}vy processes with a view towards
  mathematical finance}.
\newblock \emph{arXiv preprint math/0412351}, 2004.

\bibitem[Figueroa-L{\'o}pez and Mancini(2017)]{figueroa2017optimum}
J.~Figueroa-L{\'o}pez and C.~Mancini.
\newblock Optimum thresholding using mean and conditional mean square error.
\newblock \emph{arXiv preprint arXiv:1708.04339}, 2017.

\bibitem[Geman(2002)]{geman2002}
H.~Geman.
\newblock {Pure jump L{\'e}vy processes for asset price modelling}.
\newblock \emph{Journal of Banking \& Finance}, 2002.

\bibitem[Jacod and Rei{\ss}(2014)]{jacod2014remark}
J.~Jacod and M.~Rei{\ss}.
\newblock {A remark on the rates of convergence for integrated volatility
  estimation in the presence of jumps}.
\newblock \emph{The Annals of Statistics}, 2014.

\bibitem[Kallsen and Tankov(2006)]{kallsen2006characterization}
J.~Kallsen and P.~Tankov.
\newblock {Characterization of dependence of multidimensional L{\'e}vy
  processes using L{\'e}vy copulas}.
\newblock \emph{Journal of Multivariate Analysis}, 2006.

\bibitem[Lehmann and Romano(2006)]{lehmann2006testing}
E.~Lehmann and J.~Romano.
\newblock \emph{Testing statistical hypotheses}.
\newblock Springer Science \& Business Media, 2006.

\bibitem[Mancini(2009)]{mancini2009}
C.~Mancini.
\newblock Non-parametric threshold estimation for models with stochastic
  diffusion coefficient and jumps.
\newblock \emph{Scandinavian Journal of Statistics}, 2009.

\bibitem[Mancini(2017)]{mancini2017truncated}
C.~Mancini.
\newblock Truncated realized covariance when prices have infinite variation
  jumps.
\newblock \emph{Stochastic Processes and their Applications}, 2017.

\bibitem[Martin and Vetter(2017)]{martin2017null}
O.~Martin and M.~Vetter.
\newblock The null hypothesis of common jumps in case of irregular and
  asynchronous observations.
\newblock \emph{arXiv:1712.07159}, 2017.

\bibitem[Neumann and Rei{\ss}(2009)]{neumann2009nonparametric}
M.~Neumann and M.~Rei{\ss}.
\newblock {Nonparametric estimation for L{\'e}vy processes from low-frequency
  observations}.
\newblock \emph{Bernoulli}, 2009.

\bibitem[Nickl et~al.(2016)Nickl, Rei{\ss}, S{\"o}hl, and Trabs]{Nic}
R.~Nickl, M.~Rei{\ss}, J.~S{\"o}hl, and M.~Trabs.
\newblock {High-frequency Donsker theorems for L{\'e}vy measures }.
\newblock \emph{Probability Theory and Related Fields}, 2016.

\bibitem[Rosinski(1990)]{rosinski1990}
J.~Rosinski.
\newblock On series representations of infinitely divisible random vectors.
\newblock \emph{The Annals of Probability}, 1990.

\bibitem[Tankov(2004)]{pt}
P.~Tankov.
\newblock {Processus de L{\'e}vy en Finance: Probl{\'e}mes Inverses et
  Mod{\'e}lisation de D{\'e}pendance}.
\newblock \emph{Ph.D Thesis}, 2004.

\bibitem[Tankov and Cont(2003)]{tankov2003financial}
P.~Tankov and R.~Cont.
\newblock \emph{Financial modelling with jump processes}.
\newblock CRC press, 2003.

\bibitem[Todorov and Tauchen(2011)]{todorov2011volatility}
V.~Todorov and G.~Tauchen.
\newblock {Volatility jumps}.
\newblock \emph{Journal of Business \& Economic Statistics}, 2011.

\bibitem[Tsybakov(2009)]{tsybakov2009introduction}
A.~Tsybakov.
\newblock \emph{Introduction to Nonparametric Estimation}.
\newblock Springer Series in Statistics. Springer, New York, 2009.

\bibitem[Wu(2007)]{wu2007}
L.~Wu.
\newblock {Modeling financial security returns using L{\'e}vy processes}.
\newblock \emph{Handbooks in operations research and management science}, 15,
  2007.

\end{thebibliography}
\bibliographystyle{plainnat}

\end{document}